\documentclass[11pt]{amsart}
\usepackage[dvips]{graphicx}
\usepackage{amssymb,latexsym,amsmath}
\usepackage{amsfonts}
\usepackage{amsthm}
\usepackage{slashed}
\usepackage[all]{xy}
\usepackage{color}
\usepackage{slashed}
\usepackage{tikz}%
\usetikzlibrary{matrix,arrows}
\usepackage{float}
\usepackage{mathrsfs}
\usepackage{empheq}
\usepackage{multicol}

\numberwithin{equation}{section}

\newcounter{introequation}
\newenvironment{introequation}{\refstepcounter{introequation}\equation}{\tag{\theintroequation}\endequation}

\usepackage{newfloat}
\DeclareFloatingEnvironment[
    fileext=los,
    listname=List of Schemes,
    name=Figure,
    placement=tbhp,
    within=section,
]{scheme}

\newtheorem{theorem}{Theorem}[section]
\newtheorem{THM}{Theorem}
\newtheorem{proposition}[theorem]{Proposition}
\newtheorem{lemma}[theorem]{Lemma}
\newtheorem{coro}[theorem]{Corollary}

\theoremstyle{definition}
\newtheorem{definition}[theorem]{Definition}
\newtheorem{remark}[theorem]{Remark}
\newtheorem{example}[theorem]{Example}

\newcounter{x}

\newcommand{\norm}[1]{\lVert #1 \rVert}
\newcommand{\bnorm}[1]{\bigg{\lVert} #1 \bigg{\rVert}}

\newcommand{\inner}[2]{\langle #1,#2\rangle}

\newcommand{\ran}{\textnormal{ran}}
\newcommand{\ind}{\textnormal{ind}}

\newcommand{\spec}{\textnormal{spec}}

\newcommand{\rk}{\text{rk}}

\newcommand{\vol}{\text{vol}}

\newcommand{\dom}{\textnormal{Dom}}

\usepackage{color}
\usepackage[pdftex,colorlinks=false,a4paper,bookmarks=true,bookmarksnumbered=true,
linkcolor=blue,citecolor=blue,citebordercolor={0 0 1}]{hyperref}
\usepackage{hyperref}

\begin{document}
\title{The $S^1$-Equivariant signature for semi-free actions as an index formula}
\author{Juan Camilo Orduz}
\address{Institut F\"ur Mathematik , Humboldt Universit\"at zu
Berlin, Germany.}
\email{juanitorduz@gmail.com}
\date{\today}

\maketitle

\begin{abstract}
In \cite{L00} John Lott defined an integer-valued signature $\sigma_{S^1}(M)$ for the orbit space
of a compact orientable manifold with a semi-free $S^1$-action but he did not construct
a Dirac-type operator which has this signature as its index. 
We construct such operator on the orbit space and
we show that it is essentially unique and that its index coincides
with Lott's signature, at least when the stratified space satisfies
the so-called Witt condition. For the non-Witt case, this operator 
remains essentially self-adjoint (in contrast to the  Hodge de-Rham operator) and it has a well-defined index which 
we conjecture will also compute $\sigma_{S^1}(M)$. 
\end{abstract}

\setcounter{tocdepth}{1}
\tableofcontents

\section*{Introduction}

In \cite{L00} John Lott studied a signature invariant for quotients of closed oriented $(4k+1)$-dimensional manifolds $M$ by $S^1$-actions. This signature, denoted by $\sigma_{S^1}(M)$, is defined at the level of basic forms with compact support. If the action is semi-free, i.e. if the isotropy groups are either the trivial group or the whole $S^1$, the quotient space $M/S^1$ is a stratified space with singular stratum the fixed point set $M^{S^1}$. Locally, a neighborhood of the singular stratum is homeomorphic to the product $D^{4k-2N-1}\times C(\mathbb{C} P^N)$ for some $N$. Here $C(\mathbb{C} P^N)$ denotes a cone with link $\mathbb{C}P^N$. Moreover, on the open and dense subset $M_0\coloneqq M-M^{S^1}$  the action is free. In this context, Lott proved that the following remarkable formula yields an $S^1$ -homotopy invariant(\cite[Theorem 4]{L00})
\begin{introequation}\label{Eqn:Lott}
\sigma_{S^1}(M)=\int_{M_0/S^1}L\left(T(M_0/S^1), g^{T(M_0/S^1)}\right)+\eta(M^{S^1}),
\end{introequation}
where $L\left(T(M_0/S^1), g^{T(M_0/S^1)}\right)$ is the $L$-polynomial of the curvature form of the tangent bundle $T(M_0/S^1)$ with respect to the quotient metric $g^{T(M_0/S^1)}$and $\eta(M^{S^1})$ is the eta invariant of the odd signature operator defined on the fixed point set. It is important to emphasize that part of the result is the convergence of the integral over $M_0/S^1$. The question that arises naturally is whether there exists a Fredholm operator whose index computes $\sigma_{S^1}(M)$. This question was posed by Lott himself as a remark in his original work \cite[Section 4.2]{L00}. A natural candidate is the Hodge-de Rham operator $D_{M_0/S^1}
\coloneqq d_{M_0/S^1}+d^{\dagger}_{M_0/S^1}$ defined on the space of compactly supported differential forms $\Omega_c(M_0/S^1)$. If the quotient metric is not complete, $D_{M_0/S^1}$ might have several closed extensions. In order to understand this phenomenon better it is necesary to study the form of the operator close to the fixed point set. Following Br\"uning's work  \cite{B09} one sees that $D_{M_0/S^1}$, close to $F\subset M^{S^1}$, is unitarily equivalent to an operator of the form 
\begin{introequation}\label{Eqn:OpSing}
\Psi^{-1}{D}_{M_0/S^1}\Psi
=\gamma\left(\frac{\partial}{\partial r}+
\left(\begin{array}{cc}
I & 0\\
0 & -I
\end{array}\right)\otimes A(r)
\right).
\end{introequation}
The operator $A(r)$ can be written as
\begin{introequation}\label{Eqn:A}
A(r)\coloneqq A_H(r)+\frac{1}{r}A_V,
\end{introequation}
where $A_H(r)$ is a first order horizontal operator, well-defined for $r\geq 0$. The coefficient $A_V$ is a first order vertical operator, known as the {\em cone coefficient}. Using the techniques developed in \cite{BS91}, Br\"uning showed in \cite[Section 4]{B09} that the operator \eqref{Eqn:OpSing} has a discrete self-adjoint extension. In addition, if the cone coefficient satisfies the spectral condition
\begin{introequation}\label{Eqn:AVgeg12}
|A_V|\geq\frac{1}{2}, 
\end{introequation} 
then the operator is in fact essentially self-adjoint. In the {\em Witt case}, i.e. when there are no vertical harmonic forms in degree $N$ (i.e. $N$ odd), we can always achieve condition \eqref{Eqn:AVgeg12} by rescaling the vertical metric, which is an operation that preserves the index. To see this one needs to understand the spectrum of $A_V$. It was shown in \cite[Theorem 3.1]{B09} that the essential eigenvalues, those invariant under the rescaling, are the ones obtained when restricting to the space of vertical harmonic forms. These eigenvalues are explicitly given by $2j-N$, for $j=0,1,\cdots, N$. Observe that if $N$ is odd, zero does not appear as an essential eigenvalue and $|2j-N|\geq 1$. On the other hand, if $N=2\ell$ then zero appears as an eigenvalue when $j=\ell$ and the corresponding eigenspace is non-zero. For the Witt case we prove following the work of Br\"uning \cite[Section 5]{B09} that $\ind(D_{M_0/S^1}^+)=\sigma_{S^1}(M)$, where $D_{M_0/S^1}^+$ is the chiral Dirac operator with respect to the Clifford involution $\star_{M_0/S^1}$. \\

Up to this point the picture looks incomplete as Lott's geometric proof of \eqref{Eqn:Lott} works without any distinction on the parity of $N$. In contrast, for the analytical counterpart one needs to distinguish between  the Witt and the non-Witt case since in the latter we are forced to impose boundary conditions. This motivates the following question: Does there exist an essentially self-adjoint operator on $M/S^1$, independent of the codimension of the fixed point set in $M$, whose index is precisely the $S^1$-signature? Hope for the existence of such operator relies on the fundamental work \cite{BH78} of Br\"uning and Heintze,  where the authors develop a machinery to ``push-down" self-adjoint operators to quotients of compact Lie group actions. The key observation of their formalism is that, whenever a self-adjoint operator commuting with the group action is restricted to the space of invariant sections, it remains self-adjoint in the restricted domain. Once this result is established, Br\"uning and Heintze constructed a unitary map $\Phi$ between the space of square integrable invariant sections on the open set of principal orbits and the space of square integrable sections of a certain vector bundle defined on the quotient space. This construction seems appropriate for our case of interest because all geometric differential operators on $M$, defined on smooth forms, are essentially self-adjoint since $M$ is closed. The next question is to determine which operator to choose in order to apply Br\"uning and Heintze's construction.  Two natural candidates are the Hodge-de Rham operator and the odd signature operator. Implementing the procedure described above for these two operators, one obtains only partially satisfactory results. Concretely, the induced operators are indeed self-adjoint by construction, but the resulting potentials do not anti-commute with $\star_{M_0/S^1}$ (\cite[Section 4.3]{JO17}). This is of course a problem since $\star_{M_0/S^1}$ is the natural involution which should split the desired push down operator in order to obtain the $S^1$-signature. Nevertheless, going back to the construction of \cite{BH78}, one can see that it is enough to push down a transversally elliptic operator in order to obtain an elliptic operator on the quotient. Using this observation, which enlarges the pool of candidates for the operator, and by analyzing the concrete form of the unitary transformation $\Phi$ defined by Br\"uning and Heintze, we are able to find an essentially self-adjoint $S^1$-invariant transversally elliptic operator whose induced push-down operator satisfies the desired conditions. Indeed, consider the first order symmetric transversally elliptic differential operator $B\coloneqq -c(\chi)d+d^\dagger c(\chi):\Omega_c(M_0)\longrightarrow\Omega_c(M_0)$, where $c(\chi)$ denotes the left Clifford action on $\wedge T^*M$ by the characteristic $1$-form of the induced foliation by the $S^1$-action. As $B$ commutes with the Gau\ss-Bonnet grading $\varepsilon\coloneqq (-1)^j$ on $j$-forms, then we define $\mathscr{D}'$ through the following commutative diagram 

\begin{align*}
\xymatrixcolsep{2cm}\xymatrixrowsep{2cm}\xymatrix{
\Omega_c^\text{ev}(M_0)^{S^1} \ar[r]^-{B^\text{ev}} & \Omega_c^\text{ev}(M_0)^{S^1} \\
\Omega_c(M_0/S^1) \ar[u]^-{\psi_\text{ev}} \ar[r]^-{\mathscr{D}'}& \Omega_c(M_0/S^1), \ar[u]_-{\psi_\text{ev}}
 }
\end{align*}
where $\psi_\text{ev}$ is a  modification of the unitary transformation introduced in \cite[Section 5]{BS88}. 

\begin{THM}\label{THM1}
The operator $\mathscr{D}':\Omega_c(M_0/S^1)\longrightarrow \Omega_c(M_0/S^1)$ is given explicitly by 
\begin{align*}
\mathscr{D}'=D_{M_0/S^1} +\frac{1}{2}c(\bar{\kappa})\varepsilon-\frac{1}{2}\widehat{c}(\bar{\varphi}_0)(1-\varepsilon),
\end{align*}
where $\bar{\kappa}$ is the mean curvature form and $\widehat{c}(\bar{\varphi}_0)$ is a bounded endomorphism. In addition $\mathscr{D}'$ satisfies:
\begin{enumerate}
\item It anti-commutes with $\star_{M_0/S^1}$.
\item It is essentially self-adjoint. 
\item It has the same principal symbol as  the Hodge-de Rham operator $D_{M_0/S^1}$ .
\item It is discrete. 
\end{enumerate}
\end{THM}

By the Kato-Rellich theorem it is enough to study the essentially self-adjoint operator
\begin{align*}
\mathscr{D}\coloneqq D_{M_0/S^1}+\frac{1}{2}c(\bar{\kappa})\varepsilon.
\end{align*}
We call this operator the {\em induced Dirac-Schr\"odinger operator}. Since the mean curvature form can be written close to the fixed point set as $\bar{\kappa}=-dr/r$, one verifies that close to the fixed point set we can express similarly
\begin{introequation}\label{Eqn:OpSingPot}
\Psi^{-1}\mathscr{D}\Psi
=\gamma\left(\frac{\partial}{\partial r}+
\left(\begin{array}{cc}
I & 0\\
0 & -I
\end{array}\right)\otimes \left(A(r)-\frac{\varepsilon}{2r}\right)
\right).
\end{introequation}
One can deduce from \cite[Theorem 3.1]{B09} that
\begin{introequation}\label{CondSpecOpPot}
\spec\left(A_V-\frac{1}{2}\varepsilon\right)\cap\left(-\frac{1}{2},\frac{1}{2}\right)=\emptyset,
\end{introequation}
which verifies that $\mathscr{D}$ is indeed essentially self-adjoint. Furthermore, it is easy to verify that the parametrix's construction of \cite[Section 4]{B09} can be adapted to  $\mathscr{D}$, which allows us to prove that this operator is discrete. \\

For the Witt  case we show that the index of $\mathscr{D}^+$ computes the signature invariant $\sigma_{S^1}(M)$.
\begin{THM}\label{THM2}
In the Witt case we have for the graded Dirac-Schr\"odinger operator $\mathscr{D}^+$ the following index identity 
\begin{align*}
\ind(\mathscr{D}^+)=\sigma_{S^1}(M)=\int_{M_0/S^1}L\left(T(M_0/S^1), g^{T(M_0/S^1)}\right).
\end{align*}
\end{THM}
Here we have used the fact that in the Witt case the eta invariant of the odd signature operator of the fixed point set vanishes. \\

For the non-Witt case the index computation has been so far elusive. Nevertheless, as  $\ind(\mathscr{D}^+)$ is still defined in this case and in view of Theorem \ref{THM2}, we expect an analogous result. 

\subsection*{Acknowledgments}
I would like to thank Prof. Jochen  Br\"uning for all these years’ guidance and unconditional support — in particular, for the enlightening conversations during the development of this project. I am deeply indebted to Sara Azzali, Francesco Bei,  Batu G\"uneysu , Prof. Sylvie Paycha, Prof. Ken Richardson and  Asilya Suleymanova for all the ideas that contributed to this work.  I gratefully acknowledge the financial support from the Berlin Mathematical School and  the project  SFB 647: Space - Time - Matter.

\section{Br\"uning-Heintze contruction }\label{Sect:BH}

The aim on this first section is to give a brief description of the construction, developed in the fundamental work \cite{BH78} of Br\"uning and Heintze, of induced self-adjoint operators on quotients of compact Lie group actions. 

\subsection{The isomorphism $\Phi$ }

Let $G$ be a compact Lie group acting on a smooth oriented Riemannian manifold $M$ by orientation preserving isometries. Denote by $M^{G}\subset M$ the fixed point set and by $M_0\subset M$ the open and dense subset consisting of the union of all principal orbits in $M$ (\cite[Theorem 2.8.5]{DK00}). If $M$ is connected then so is $M_0$.  
For $x\in M$, we denote by $Gx$ and $G_x$ its {\em orbit} and its {\em isotropy group} respectively. The {\em orbit map} $\pi_G:M\longrightarrow M/G$  induces a Riemannian structure on the orbit space $M/G$ by requiring $\pi_G$ to be a Riemannian submersion. We call it the {\em quotient metric}. In addition, let $\pi_E: E\longrightarrow M$ be a complex vector bundle with Hermitian metric $\inner{\cdot}{\cdot}_E$.  This metric induces an inner product on the space of continuous sections with compact support $C_c(M,E)$ by
\begin{equation}\label{Eqn:HermE}
(s, s')_{L^2(E)}\coloneqq \int_M \inner{s(x)}{s'(x)}_E \vol_M(x),
\end{equation}
where $s,s'\in C_c(M,E)$, $x\in M$ and $\vol_M$ denotes the Riemannian volume element on $M$. Define the space $L^2(E)$  as the Hilbert space completion of $C_c(M,E)$ with respect to the inner product \eqref{Eqn:HermE}.
\begin{remark}\label{Rmk:M-M0}
By \cite[Proposition IV.3.7]{B72} it follows  that $M-M_0$ has measure zero with respect to the Riemannian measure, hence $L^2(E)=L^2(E|_{M_0})$. 
\end{remark}
Assume further that $E$ is a  {\em $G$-equivariant vector bundle}, i.e. the projection $\pi_E$ commutes with a $G$-action on $E$. In this context there is an induced action of $G$ on the space of continuous sections $C(M,E)$ defined by the relation
\begin{equation}\label{Eqn:GActionSections}
(U_gs)(x)\coloneqq g(s(g^{-1}x)),
\end{equation}
where $g\in G$ $s\in C(M,E)$ and  $x\in M$. This action induces a unitary representation of $G$ in $L^2(E)$. We say that a section $s\in C(M,E)$ is {\em $G$-invariant} if $U_g s=s$ for all $g\in G$ and we denote by $C(M,E)^G$ and $L^2(E)^G$ the $G$-invariant subspaces of $C(M,E)$ and $L^2(E)$ respectively. 

\begin{example}[Exterior Algebra]\label{Ex:ExtAlg}
As before let $M$ be an oriented Riemannian manifold on which $G$ acts by orientation preserving isometries. The action on $M$ induces an action on the exterior algebra bundle  $E=\wedge_\mathbb{C} T^*M\coloneqq \wedge T^*M \otimes\mathbb{C}$ so that $E$ becomes a $G$-vector bundle over $M$. The action \eqref{Eqn:GActionSections} on differential forms is simply given by $U_g\omega=(g^{-1})^{*}\omega$. 
\end{example}

Now consider the subset
\begin{equation}\label{Eqn:DefE'}
E'\coloneqq \bigcup_{x\in M_0} E^{G_x}_x
\end{equation}
where $E^{G_x}_x$ denotes the elements of the fiber $E_x\coloneqq\pi_E^{-1}(x)$ which are invariant under the $G_x$-action.
If $M$ is connected then $E'$ is a $G$-equivariant subbundle of $E\big{|}_{M_0}$ \cite[Lemma 1.2]{BH78}. As $G$ acts on $E'$ with one orbit type, then it follows that  $F\coloneqq E'/G$ is a manifold (\cite[Theorem 2.6.7]{DK00}). In addition, if $\pi'_G:E'\longrightarrow F$ denotes the orbit map and $\pi_{E'}:E'\longrightarrow M_0$ denotes the projection, then using the Slice Theorem (\cite[Theorem 2.4.1]{DK00}) one can show that $F$ is a vector bundle over $M_0/G$, of the same rank as $E'$,  and the following diagram commutes:
\begin{align}\label{Diag:OrbitMaps}
\xymatrixrowsep{2cm}\xymatrixcolsep{2cm}\xymatrix{
E' \ar[d]_-{\pi_{E'}}  \ar[r]^-{\pi'_G}  & F \ar[d]^-{\pi_{F}} \\
M_0 \ar[r]^-{\pi_G} & M_0/G.
}
\end{align}
\begin{lemma}[{\cite[Lemma 1.13]{JO17}}]\label{Lemma:InducedMetricF}
The bundle $F$ inherits a Hermitian metric $\inner{\cdot}{\cdot}_F$ from $E'$ defined by the relation 
$\inner{\pi'_G v_1}{\pi'_G v_2}_F(y)\coloneqq \inner{v_1}{v_2}_E(x),$ where $x\in M_0$, $v_1,v_2\in E'_x$ and $\pi_G(x)=y$.
\end{lemma}

For $y\in M_0/G$ let $h(y)\coloneqq \vol(\pi_G^{-1}(y))$ be the {\em volume of the orbit} containing a point in $\pi_G^{-1}(y)$. We can consider the weighted inner product on $C_c(M_0/G,F)$ defined by the formula
\begin{equation}\label{Def:L2FhGen}
(s,s')_{L^2(F,h)}\coloneqq \int_{M_0/G}\inner{s(y)}{s'(y)}_F h(y) \vol_{M_0/G}(y),
\end{equation}
where  $\vol_{M_0/G}$ denotes the Riemannian volume element of $M_0/G$ with respect to the quotient metric.  Analogously we define $L^2(F,h)$ to be the completion of $C_c(M_0/G,F)$ with respect to this inner product. We now state one of the most important results of \cite{BH78}. The spirit of the proof relies on the construction above  and on Remark \ref{Rmk:M-M0}.

\begin{theorem}[{\cite[Theorem $1.3$]{BH78}}]\label{Thm:Fund}
There is an isometric isomorphism of Hilbert spaces
$$\Phi:L^2(E)^{G}\longrightarrow L^2(F,h).$$
With $\pi'_G:E'\longrightarrow F$ denoting the orbit map $\Phi$ is given by
$$\Phi s_1 \circ \pi_G(x)=\pi'_G\circ s_1(x),$$
where $s_1\in C_c(M_0,E)^G$ and $x\in M_0$. Its inverse map is given by 
$$\Phi^{-1}s_2(x)=s_2\circ \pi_G(x)\cap E_x,$$
where $s_2\in C_c(M_0/G,F)$ and $x\in M_0$. 
\end{theorem}

\subsection{Induced operators on the principal orbit type}\label{Section:InduedOperatorsGen}

In the context of the subsection above, consider a self-adjoint operator $R:\dom(R)\subset L^2(E)\longrightarrow L^2(E)$ commuting with the $G$-action, that is, $U_g(\dom(R))\subset \dom(R)$ and $U_g R(s)=RU_g(s)$ for all $g\in G$ and $s\in\dom(R)$.

\begin{lemma}[{\cite[Lemma $2.2$]{BH78}}]\label{Lemma:OpS}
The operator $S\coloneqq R|_{\dom(R)\cap L^2(E)^G}$ is a well-defined self-adjoint operator on $L^2(E)^G$ with $\dom(S)=\dom(R)^G$.
\end{lemma}

From Theorem \ref{Thm:Fund} and Lemma \ref{Lemma:OpS} we deduce the following remarkable result which allows to ``push-down" a self-adjoint operator commuting to $G$ to the quotient. 
\begin{proposition}[{\cite[pg. 178-179]{BH78}}]\label{Prop:OpT}
The operator $T:\dom(T)\subseteq L^2(F,h)\longrightarrow L^2(F,h)$  defined by
$
T \coloneqq \Phi\circ S\circ \Phi^{-1}\big{|}_{\Phi(\dom(S))},
$
with $\dom(T)\coloneqq \Phi(\dom(S))$ is a self-adjoint. 
\end{proposition}

A particular case of interest is when $R$ is generated by an elliptic differential operator $D:C_c^\infty(M,E)\longrightarrow C_c^\infty(M,E)$ so that $C^\infty_c(M,E)\subset\dom(R)$ and $R|_{C^\infty_c(M,E)}=D$. The following result states that in this case the induced operator $T$ of Proposition \ref{Prop:OpT} is also generated by a differential operator of the same order. 
\begin{proposition}[{\cite[Theorem 2.4]{BH78}}]\label{Prop:SDiff}
If $R$ is generated by a differential operator $D$ of order $k$, then $T$ is also generated by a certain differential operator $D'$ of order $k$. Their principal symbols are related by the formula
$
\sigma_P(D')(y,\xi)(\pi_G'(e))=\pi_G'(\sigma_P(D)(x,\pi_G^*\xi)(e)),
$
where $y\in M_0/G$, $\xi\in T^*_y (M_0/G)$, $x\in\pi^{-1}_G(y)$ and $e\in E'_x$. In particular, the operator $D'$ is elliptic if $D$ is transversally elliptic. 
\end{proposition}

\section{Lott's $S^1$-equivariant signature formula}\label{Sect:Lott}

\subsection{Definition of the equivariant $S^1$-signature}\label{Section:BasicsDef}

Let $(M, g^{TM})$ be an $4k+1$ dimensional, closed, oriented Riemannian manifold on which the circle $S^1$ acts by orientation-preserving isometries. Let us denote by $V$ the generating vector field of the action and by $\imath:M^{S^1}\longrightarrow M$  the inclusion of the fixed point set into $M$. We can define two sub-complexes of the de Rham complex of $M$,
\begin{align*}
\Omega_\text{bas}(M)\coloneqq &\{\omega\in\Omega(M)\:|\:L_V\omega=0\:\:\text{and}\:\:\iota_V\omega=0\},\\
\Omega_\text{bas}(M,M^{S^1})\coloneqq &\{\omega\in\Omega(M)_\text{bas}\:|\: \imath^*\omega=0\},
\end{align*}
where $L_V$ is the Lie derivative. Denote by $H^*_\text{bas}(M)$ and $H^*_\text{bas}(M,M^{S^1})$ their respective  cohomology groups.  It can be shown that there exist isomorphisms (\cite[Proposition 1]{L00})
\begin{align}\label{IsomsBasic}
H^{*}_\text{bas}(M,M^{S^1})\cong H^{*}_{\text{bas},c}(M-M^{S^1})\cong H^{*}(M/S^1,M^{S^1};\mathbb{R}),
\end{align}
where the subscript $c$ denotes cohomology with compact support. These  cohomology groups are all $S^1$-homotopy invariant (\cite[Proposition 2]{L00}).\\

Using the musical isomorphims induced by the Riemannian metric one defines the $1$-form $\alpha$ on $M-M^{S^1}$ by
 $\alpha\coloneqq V^\flat /\norm{V}^2$ so that $\alpha(V)=1$. Here we list some important properties of $\alpha$.

\begin{proposition}[{\cite[Section 2]{L00}}]\label{Prop:alpha}
The following relations hold true:
\begin{enumerate}
\item The $2$-form $d\alpha$ is basic. 
\item The form $\alpha$ satisfies $L_V\alpha=0$, that is, $\alpha$ is $S^1$-invariant.
\item If $\omega\in\Omega_{\textnormal{bas},c}^{4k-1}(M-M^{S^1})$ then 
$$\int_{M}\alpha\wedge d\omega=0. $$
\end{enumerate}
\end{proposition}

\begin{definition}[{\cite[Definition 4]{L00}}]
The {\em  equivariant $S^1$-signature}  $\sigma_{S^1}(M)$ of $M$ with respect to the $S^1$-action  is defined as the signature of the symmetric quadratic form
\begin{align*}
\xymatrixcolsep{2cm}\xymatrixrowsep{0.01cm}\xymatrix{
H^{2k}_{\text{bas},c}(M-M^{S^1})\times H^{2k}_{\text{bas},c}(M-M^{S^1}) \ar[r] & \mathbb{R}\\
(\omega,\omega') \ar@{|->}[r] & \displaystyle{\int_M\alpha\wedge \omega\wedge\omega'}.
}
\end{align*}
\end{definition}

\begin{remark}
It was shown by Lott that $\sigma(M)_{S^1}$ is independent of the Riemannian metric (\cite[Proposition 5]{L00}) and that if $f:M\longrightarrow N$ is a orientation-preserving $S^1$-homotopy equivalence then $\sigma_{S^1}(M)=\sigma_{S^1}(N)$ (\cite[Proposition 6]{L00}). 
\end{remark}

\subsection{The equivariant $S^1$-signature formula for semi-free actions}\label{Section:S1 Signature}

Let us assume now that the action is semi-free, which means that the isotropy groups are either the trivial group or the whole $S^1$. Let $M_0\coloneqq M-M^{S^1}$ be the set of principal orbits where the action is free. We equip the manifold $M_0/S^1$ with the quotient metric $g^{T(M_0/S^1)}$. In this context the dimension of the fixed point set $M^{S^1}$ must be odd, so we can consider its associated  odd signature operator (\cite[Equation 4.6]{APSI}) and the corresponding eta invariant $\eta(M^{S^1})$ (\cite[Equation 1.7]{APSI}) . One of the most important results in \cite{L00} is the following {\em index-like formula} for the equivariant $S^1$-signature. 
\begin{theorem}[{\cite[Theorem 4]{L00}}]\label{Thm:S1SignatureThm}
Suppose $S^1$ acts effectively and semifreely on $M$, then 
$$\sigma_{S^1}(M)=\int_{M_0/S^1}L\left(T(M_0/S^1),g^{T(M_0/S^1)}\right)+\eta(M^{S^1}), $$
where $L\left(T(M_0/S^1),g^{T(M_0/S^1)}\right)$ denotes the $L$-polynomial in the curvature of $g^{T(M_0/S^1)}$.
\end{theorem} 

\begin{remark}
It is important to emphasize that part of the conclusion of Theorem \ref{Thm:S1SignatureThm} is the convergence of the integral of the $L$-polynomial over the open manifold $M_0/S^1$. 
\end{remark}

\subsubsection{Sketch of the proof of the signature formula} \label{Sect:PfoofSignatureFormula}

We now give a brief description of the proof of Theorem \ref{Thm:S1SignatureThm} presented in \cite[Section 2.3]{L00}. Let $F\subset M^{S^1}$ be a connected component of the fixed point set. As the action is orientation preserving the dimension of $F$ must be odd and can be written as $\dim F = 4k -2N-1$, for some $N\in\mathbb{N}_0$. Let $NF$ denote the normal bundle of $F$ in $M$. The first step of the proof is to model a neighborhood of $F$ in $M/S^1$ as the mapping cylinder $C(\mathcal{F})$ of the projection of a Riemannian fiber bundle $\pi_{\mathcal{F}}:\mathcal{F}\longrightarrow F$ where $\mathcal{F}\coloneqq \mathcal{S}/S^1$ and $ \mathcal{S}\coloneqq SNF$ denotes the associated sphere bundle (Figure \ref{Fig:LocalDescrF}). By \cite[Lemma 2.2]{U70} it follows that the fibers of $\mathcal{F}$ are copies of $\mathbb{C}P^N$. 

\begin{figure}[h]
\begin{center}
\begin{tikzpicture}
\draw (-2,5)--(2,5);
\draw (-3,3)--(1,3);
\draw (-2,5)--(-3,3);
\draw (2,5)--(1,3);
\draw (-2.5,1)--(1.5,1);
\draw (-3,3)--(-2.5,1);
\draw (2,5)--(1.5,1);
\draw (1,3)--(1.5,1);
\draw [dashed](-2,5)--(-2.25,3);
\draw (-2.5,1)--(-2.25,3);
\draw (-1,3)--(-0.5,1);
\draw [dashed](0,5)--(-0.25,3);
\draw (-0.5,1)--(-0.25,3);
\draw (0,5)--(-1,3);
\draw (-2.4,2) arc (87:155:0.3 and 0.5);
\draw (1.6,2) arc (87:155:0.3 and 0.5);
\draw (-0.4,2) arc (87:155:0.3 and 0.5);
\draw (-2.7,1.7)--(1.33,1.7);
\draw (-2.4,2)--(1.62,2);
\node (a) at (-1,4) {$\mathbb{C}P^N$};
\node (b) at (0,0.7) {$F$};
\draw [<-](-2.84,1.68)--(-2.65,1);
\node (c) at (-3,1.3) {$r$};
\node (d) at (-2.5,0.8) {$0$};
\node (e) at (-3.2,3) {$1$};
\end{tikzpicture}
\caption{Mapping cylinder of the $\mathbb{C}P^N$-fibration $\pi_{\mathcal{F}}:\mathcal{F}\longrightarrow F.$}\label{Fig:LocalDescrF}
\end{center}
\end{figure}
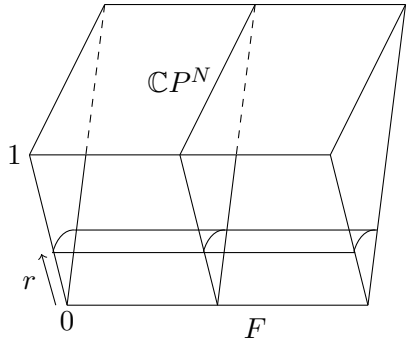

For $t>0$, let $N_t(F)$ be the $t$-neighborhood of $F$ in $M/S^1$.  Using \cite[Proposition A.III.5]{BGM}) one can verify $\sigma_{S^1}(M)=\sigma(M/S^1-N_{t}(F))$ for $r$ small enough. Applying the Atiyah-Patodi-Singer signature theorem we get
\begin{align}\label{Eq:MainIdea}
\sigma( M/S^1 - N_t(F))=&\int_{M/S^1-N_r(F)}L\left(T(M/S^1 - N_t(F)),g^{T(M/S^1 - N_t(F))}\right)\\
& +\int_{\partial N_t(F)}TL(\partial N_t(F))+\eta(\partial N_t(F))\notag.
\end{align}
The main idea of Lott's proof is to study the behavior of the terms in \eqref{Eq:MainIdea} as $r\longrightarrow 0$. Let us see how to do this for the first term. Let $\{e^i\}_{i=1}^{2N}\cup\{f^\alpha\}_{\alpha=1}^{4k-2N-1}$ be a local orthonormal basis for $T^*{\mathcal{F}}$ as in \cite[Section III(c)]{BL95}. Let $\omega$ be the connection $1$-form of the Levi-Civita connection associated to this basis. The components of $\omega$ satisfy the structure equations 
\begin{align*}
de^i+\omega^i_j\wedge e^j+\omega^i_\alpha\wedge f^\alpha=0,\\
df^\alpha+\omega^\alpha_j\wedge e^j+\omega^\alpha_\beta\wedge f^\beta=0.
\end{align*}
Let $\Omega$ denote the associated curvature $2$-form.  We can construct a local orthonormal basis for $T^*C(\mathcal{F})$ from the basis above as $\{dr\}\cup\{\widehat{e}^i\}_{i=1}^{v}\cup\{\widehat{f}^\alpha\}_{\alpha=1}^{h}$ for $0<r<t$ where $\widehat{e}^i\coloneqq re^i$ and $
\widehat{f}^\alpha\coloneqq f^\alpha$. The associated connection $1$-form $\widehat{\omega}$  satisfies the structure equations 
\begin{align*}
\widehat{\omega}^r_i\wedge \widehat{e}^i+\widehat{\omega}^r_\alpha\wedge \widehat{f}^\alpha=&0,\\
d\widehat{e}^i+\widehat{\omega}^i_j\wedge \widehat{e}^j+\widehat{\omega}^j_\alpha\wedge \widehat{f}^\alpha+\widehat{\omega}^i_r\wedge dr=&0,\\
d\widehat{f}^\alpha+\widehat{\omega}^\alpha_j\wedge \widehat{e}^j+\widehat{\omega}^\alpha_\beta\wedge \widehat{f}^\beta++\widehat{\omega}^\alpha_r\wedge dr=&0.
\end{align*}
It can be shown that the components of $\widehat{\omega}$ are
\begin{align}\label{Eqns:Conn1From}
\widehat{\omega}^{i}_j &= \omega^i_j,\nonumber\\
\widehat{\omega}^{i}_r  &= e^i,\nonumber\\
\widehat{\omega}^{i}_\alpha &= r\omega^i_\alpha ,\\
\widehat{\omega}^{\alpha}_\beta &= \omega^{\alpha}_{\beta\gamma}f^\gamma+r^2\omega^{\alpha}_{\beta j}e^j,\nonumber \\
\widehat{\omega}^{\alpha}_r &=0 .\nonumber  
\end{align}
Using these expressions one can compute the components of the curvature $2$-form $\widehat{\Omega}$ as $r\longrightarrow 0$ to show that the limit of the first term in the right hand side of \eqref{Eq:MainIdea} is well-defined. 
The transgression term in \eqref{Eq:MainIdea} vanishes using equivariant methods (\cite[Proposition 7.35]{BGV}. For the limit of the eta invariant one uses Dai's formula (\cite[Theorem 0.3]{D91}). One can verify that both the eta form (\cite[Sections 9.4, 10.7]{BGV}, \cite[Section 1.c]{G00}) and the tau invariant (\cite[pg. 170, 270]{BT82}) vanish. 

\subsubsection{The Witt condition} Now we comment on an important topological interpretation of the $S^1$-signature in the context of intersection homology theory introduced by Goresky and MacPherson in \cite{GMcP80}. As we saw in the proof of Theorem \ref{Thm:S1SignatureThm} above, each connected component $F\subset M^{S^1}$ of the fixed point set is an odd-dimensional closed manifold, whose dimension can be written as $\dim F=4k-2N-1$. We now distinguish the two possible cases for $N$. We say that $M/S^1$ satisfies the {\em Witt condition}, if $N$ is odd, that is, the codimension of the fixed point set $M^{S^1}$ in $M$ is divisible by four. Note that in this case $\eta(M^{S^1})=0$. This follows from \cite[Remark (3),(4), pg. 61]{APSI} since $4k-2N-1=2(2k-N)-1$ and $2k-N$ is odd if and only if $N$ is odd. Stratified spaces satisfying that the middle dimensional cohomology of the links vanish are called {\em Witt spaces} ( \cite{S83}). This is of course consistent with the definition above since $H^N(\mathbb{C}P^N)$ vanishes if and only if $N$ is odd. For this kind of stratified spaces one can always construct the Goresky-MacPherson $L$-class following the procedure described in detail in \cite[Section 5.3]{B07}, for example.  The following result describes an explicit form of the $L$-homology class for our case of interest.  

\begin{proposition}[$L$-homology Class, {\cite[Proposition 8]{L00}}]\label{Prop:LHomologyClass}
In the Witt case the differential form $L(T(M/S^1))$ represents the $L$-homology class of $M/S^1$.
\end{proposition}

In addition, for Witt spaces there is a well-defined non-degenerate  pairing in intersection homology (\cite[Section 4.4]{B07}) which gives rise to a signature invariant. 

\begin{coro}[{\cite[Corollary 1]{L00}}]\label{Coro:IH}
In the Witt case, 
\begin{align*}
\sigma_{S^1}(M)=\int_{M_0/S^1} L(T(M_0/S^1)),
\end{align*}
equals the intersection homology signature of $M/S^1$. 
\end{coro}

\section{Induced Dirac-Schr\"odinger operator on $M_0/S^1$}\label{Section:InducedOpConstruction}\label{Sec:Induced}

The goal of this section is to implement the construction described in Section \ref{Sect:BH} for the special case of a semi-free $S^1$-action, as discussed in  Section \ref{Section:S1 Signature}. 

\subsection{The mean curvature $1$-form} \label{Sect:MeanCurvFrom}

Let $M$ be an $(n+1)$-dimensional oriented, closed Riemannian manifold on which $S^1$ acts by orientation preserving isometries (before we had $n=4k$, but here we treat the general case). Denote by $\nabla$ its associated Levi-Civita connection. As in Section \ref{Section:BasicsDef} let $V$ be the generating vector field of the $S^1$-action. The flow of the vector field $V$ generates a $1$-dimensional foliation $L$ on $M_0$. We will denote by $L^\perp$ the {\em transverse distribution}, i.e.  $TM_0=L\oplus L^\perp$. The distribution $L$ is always integrable and the corresponding integral curves are precisely the $S^1$-orbits. In contrast, the transverse distribution is not necessarily integrable. 

\begin{definition}
Let $X\coloneqq V/\norm{V}$ be the {\em unit vector field} which defines the foliation $L$.  Using the musical isomorphism induced by the metric we define (see \cite{T97})
\begin{enumerate}
\item The associated {\em characteristic $1$-form} $\chi\coloneqq X^\flat$.
\item The {\em mean curvature vector field} $H\coloneqq\nabla_X X$.
\item The {\em mean curvature $1$-form} $1$-form $\kappa\coloneqq H^\flat$.
\end{enumerate}
\end{definition}

It is easy to see that the mean curvature vector field satisfies $H\in C^\infty (L^\perp)$. As a consequence $\kappa$ is horizontal, i.e. $\iota_X\kappa=0$.

\begin{lemma}[{\cite[Chapter 6]{T97}}]\label{Lemma:kappa}
The mean curvature form $\kappa$ satisfies $\kappa=L_X \chi$. 
\end{lemma}

Let $\alpha\coloneqq V^\flat/\norm{V}^2$  be the $1$-form considered in Section  \ref{Section:BasicsDef}.
The following statement is a consequence of Proposition \ref{Prop:alpha}(2) and Lemma \ref{Lemma:kappa}.

\begin{coro}\label{Coro:ChiInv}
The characteristic $1$-form $\chi$ satisfies $L_V\chi=0$, that is $\chi$ is $S^1$-invariant. 
\end{coro}

We can combine Cartan's formula and Lemma \ref{Lemma:kappa}  to get  $\kappa=\iota_X d\chi$, or equivalently 
\begin{equation}\label{Eqn:dchi}
d\chi + \kappa\wedge \chi \eqqcolon \varphi_0,
\end{equation}
where $\varphi_0$ satisfies $\iota_X\varphi_0=0$, i.e. $\varphi_0$ is {\em horizontal}. Equation \eqref{Eqn:dchi} is known as {\em Rummler's formula} and it holds for general tangentially oriented foliations (\cite[Lemma 10.4]{BGV}, \cite[Chapter 4]{T97}). Observe that the characteristic form $\chi$ can be thought as the volume form on each leaf of the foliation as the vector field $X$ satisfies $X\in C^\infty (L)$, it is $S^1$-invariant and $\norm{\chi}=1$. In particular, the volume of the orbit function $h:M_0/S^1\longrightarrow \mathbb{R}$ used in \eqref{Def:L2FhGen} can be written explicitly as  
\begin{equation}\label{Eqn:h}
h(y)=\int_{\pi^{-1}_{S^1}(y)} \chi.
\end{equation}
\begin{lemma}\label{Lemma:dh}
The exterior derivative of the volume of the orbit function is $dh=-h\kappa$. Thus, the mean curvature form $\kappa$ measures the volume change of the orbits. 
\end{lemma}
\begin{proof}
We use  \cite[Proposition 6.14.1]{BT82} and  \eqref{Eqn:dchi} to compute 
\begin{align*}
d h= d\int_{\pi^{-1}_{S^1}(y)} \chi= \int_{\pi^{-1}_{S^1}(y)}d\chi=-\int_{\pi^{-1}_{S^1}(y)}\kappa\wedge\chi +\int_{\pi^{-1}_{S^1}(y)}\varphi_0=-h\kappa, 
\end{align*}
where we have used that the integral of $\varphi_0$ is zero because this is a horizontal $2$-form. 
\end{proof}
The next proposition shows that all the geometric quantities discussed above are encoded in the norm of the generating vector field $V$. 
\begin{proposition}[{\cite[Proposition 4.7]{JO17}}]\label{Prop:NormV}
In terms of $\norm{V}$ we can express
\begin{enumerate}
\item $\chi=\norm{V}\alpha$.
\item $\kappa=-d\log(\norm{V})$.
\item $\varphi_0=\norm{V} d\alpha$.
\end{enumerate}
\end{proposition}

\begin{coro}\label{Coro:kappa}
The mean curvature $1$-form $\kappa$ is closed and basic.
\end{coro}

We end this subsection with some properties of the $2$-from $\varphi_0$. 

\begin{proposition}[{\cite[Proposition 4.9]{JO17}}]\label{Prop:varphi0}
The following relations for $\varphi_0$ hold:
\begin{enumerate}
\item If $Y_1, Y_2\in C^{\infty}(L^\perp)$, then $\varphi_0(Y_1,Y_2)=-\chi([Y_1,Y_2])$.
\item $d\varphi_0+\kappa\wedge\varphi_0=0$.
\item $\varphi_0\in\Omega^2_\textnormal{bas}(M_0)$.
\end{enumerate}
\end{proposition}

\subsection{The operator $T(D)$}

Let us consider now the Hermitian $S^1$-equivariant vector bundle $E\coloneqq  \wedge_\mathbb{C} T^* M$ of Example \ref{Ex:ExtAlg}.
Recall that the action on differential forms is given by the pullback 
$U_g\omega\coloneqq  (g^{-1})^*\omega$ for $g\in S^1$. As the Hodge star operator on $M$ commutes with the $S^1$-action on differential forms (\cite[Lemma 4.11]{JO17}),  then by Proposition \ref{Prop:Chirl}(4) we obtain the following known result.

\begin{proposition}\label{Prop:S1commD}
The Hodge-de Rham operator $D=d+d^\dagger$ of $M$ defined on the core  of differential forms $\Omega(M)\coloneqq C^{\infty}(M, \wedge_\mathbb{C} T^* M )$ is $S^1$-invariant.
\end{proposition}
This shows that we are in position to apply the construction of Br\"uning and Heintze described in Section \ref{Sect:BH}. The strategy is then as  follows:
\begin{enumerate}
\item Construct explicitly the vector bundle $F\longrightarrow M_0/S^1$ and describe the $L^2$-inner product \eqref{Def:L2FhGen}. 
\item Understand the isomorphism $\Phi$ of Theorem \ref{Thm:Fund}.
\item Describe the self-adjoint operator $D:\Omega(M)^{S^1}\longrightarrow\Omega(M)^{S^1}$ of Lemma \ref{Lemma:OpS} 
\item Explicitly compute the self-adjoint operator $T$ of Proposition \ref{Prop:OpT} and describe its properties. For example, compute its principal symbol (Proposition \ref{Prop:SDiff}). 
\end{enumerate}

\begin{remark}[{\cite[Chapter II.5]{LM89}}]
The Hodge-de Rham operator is the associated Dirac operator of the Clifford bundle $\wedge_\mathbb{C}T^*M$ with left Clifford action $c(\omega)\coloneqq\omega\wedge-\iota_{\omega^\sharp}$, which satisfies the relations $c(\omega)^2=-\norm{\omega}^2$ and $c(\omega)^\dagger=-c(\omega)$. The corresponding right Clifford action is $\widehat{c}(\omega)\coloneqq\omega\wedge+\iota_{\omega^\sharp}$. 
\end{remark}

\subsubsection{Decomposition of $S^1$-invariant differential forms}
We begin with a decomposition result of the space of $S^1$-invariant forms in terms of the basic forms. Recall that we have the inclusion  $\Omega_\text{bas}(M_0)\subset \Omega(M_0)^{S^1}$,  from the Lie derivative vanishing condition.

\begin{proposition}[{\cite[Corollary 3.14]{JO17}}]\label{Coro:DecInvForm}
Any $S^1$-invariant form $\omega\in \Omega(M_0)^{S^1}$ can be uniquely decomposed as $\omega=\omega_0+\omega_1\wedge\chi$, where $\omega_0,\omega_1\in\Omega_{\textnormal{bas}}(M_0)$. With respect to this decomposition we will represent the form $\omega$ as the column vector
\begin{equation*}
\omega=
\left(
\begin{array}{c}
\omega_0\\
\omega_1
\end{array}
\right).
\end{equation*}
\end{proposition}

\subsubsection{Construction of the bundle $F$}\label{Section:ConstrF}

We start by pointing out some important remarks:  
\begin{itemize}
\item The action on  $M_0$ is free and therefore the $S^1$-invariant bundle $E'$ of \eqref{Eqn:DefE'} is nothing else but $E'=\wedge_\mathbb{C} T^*M_0$.
\item As a consequence, by counting dimensions, we see that the rank of $F$ must agree with the rank of $E'$, which is $\rk(E')= 2^{n+1}$.
\item From \cite[Lemma 6.44]{M00} it follows that for each basic form $\beta\in\Omega^r_\text{bas}(M_0)$ there exists a unique $\bar{\beta}\in\Omega^r(M_0/S^1)$ such that $\pi^*_{S^1}\bar{\beta}=\beta$. Thus, using Corollary \ref{Coro:DecInvForm} we can identify $\Omega(M_0)^{S^1}\cong \Omega(M_0/S^1)\otimes\mathbb{C}^2$, via the orbit map $\pi_{S^1}$.
\end{itemize}
These observations indicate that $F\coloneqq E'/S^1=\wedge_\mathbb{C} T^*(M_0/S^1)\oplus \wedge_\mathbb{C} T^*(M_0/S^1)$. 
Indeed, given $x\in M_0$ and $\omega_x=\omega'_x+\omega''_x\wedge\chi_x\in \wedge_\mathbb{C} T^*_x M_0$ where $\iota_{V_x}\omega'_x=\iota_{V_x}\omega''_x=0$, the orbit map on $E'$ is explicitly given by
\begin{align}\label{Def:BundleF}
\pi'_{S^1}:
\xymatrixrowsep{0.01cm}\xymatrixcolsep{2cm}\xymatrix{
E'=\wedge_\mathbb{C} T^*M_0 \ar[r] & F=\wedge_\mathbb{C} T^*(M_0/S^1)\oplus \wedge_\mathbb{C} T^*(M_0/S^1)\\
\omega_x=\omega'_x+\omega''_x\wedge\chi_x \ar@{|->}[r] & 
{\left(\begin{array}{c}
\bar{\omega}_y' \\
{\bar{\omega}_y''}
\end{array}\right),}
}
\end{align}
where $\pi_{S^1}(x)=y$ and the form $\bar{\omega}_y'\in\wedge_\mathbb{C} T^*(M_0/S^1)$ (similarly for $\bar{\omega}_y''$) is defined by the relation $\omega_x(v_x)=\bar{\omega}_y((\pi_{S^1})_* v_x)$ for all $v_x\in T_x M_0$. Hence, the diagram \eqref{Diag:OrbitMaps}  becomes,
\begin{align*}
\xymatrixrowsep{2cm}\xymatrixcolsep{2cm}\xymatrix{
E' =\wedge_\mathbb{C} T^*M_0\ar[d]_-{\pi_E} \ar[r]^-{\pi'_{S^1}} & F=\wedge_\mathbb{C} T^*(M_0/S^1)\oplus \wedge_\mathbb{C} T^*(M_0/S^1)  \ar[d]^-{\pi_F}\\
M_0 \ar[r]^-{\pi_{S^1}}& M_0/S^1.
}
\end{align*}

\subsubsection{Description of the isomorphism $\Phi$}

Given an $S^1$-invariant form with compact support $\omega\in\Omega_c(M_0)^{S^1}$  there are two unique compactly supported basic differential forms $\omega_0,\omega_1\in\Omega_{\text{bas},c}(M_0)$ such that  $\omega=\omega_0+\omega_1\wedge\chi$. With respect to the vector notation introduced in Corollary \ref{Coro:DecInvForm} we write
\begin{align*}
\left(\begin{array}{c}
\omega_0\\
\omega_1\\
\end{array}\right)=
\left(\begin{array}{c}
\pi_{S^1}^*\bar{\omega}_0\\
\pi_{S^1}^*\bar{\omega}_1\\
\end{array}\right),
\end{align*}
where $\bar{\omega}_0,\bar{\omega}_1\in\Omega_c(M_0/S^1)$. This representation allows us to express the isomorphism $\Phi$, on compactly supported forms,  as

\begin{align*}
\Phi: \xymatrixrowsep{0.01cm}\xymatrixcolsep{1.7cm}\xymatrix{
\Omega_c(M_0)^{S^1} \ar[r] & \Omega_{\text{bas},c}(M_0)\oplus \Omega_{\text{bas},c}(M_0)  \ar[r] & \Omega_c(M_0/S^1)\oplus \Omega_c(M_0/S^1)\\
\omega  \ar@{|->}[r] &
{
\left(\begin{array}{c}
\omega_0\\
\omega_1\\
\end{array}\right)=
\left(\begin{array}{c}
\pi_{S^1}^*\bar{\omega}_0\\
\pi_{S^1}^*\bar{\omega}_1\\
\end{array}\right)
}
\ar@{|->}[r] &
{
\left(\begin{array}{c}
\bar{\omega}_0\\
\bar{\omega}_1\\
\end{array}\right).
}
}
\end{align*}
We can extend this map to $\Phi:L^2(M)^{S^1}\longrightarrow L^2(F,h)$ by density. 

\subsubsection{Description of the operator $S(D)$}
Now we want to understand the operator $S$ of Lemma \ref{Lemma:OpS} associated to the Hodge-de Rham operator $D=d+d^\dagger$. First, recall that the formal adjoin of the exterior derivative can be written as  $d^\dagger =(-1)^{n}\star d\star$, where $\star$ is the {\em chirality involution}, associated to the Clifford bundle $\wedge_{\mathbb{C}}T^*M$, which is defined on $j$-forms by $\star\coloneqq i^{[(m+1)/2]+2mj+j(j-1)}*$ (\cite[Lemma 3.17]{BGV}).
\begin{proposition}[{\cite[Proposition 3.58]{BGV}}]\label{Prop:Chirl}
The chirality operator $\star$ satisfies:
\begin{enumerate}
\item $\star^2=1$.
\item $\star^\dagger =\star$.
\item For $\alpha\in T^*M$, $\star (\alpha \wedge)\star=(-1)^{n+1}\iota_{\alpha^\sharp}$.
\item $d^\dagger =(-1)^{n}\star d\star$.
\item $\star c(\omega)=(-1)^{n}c(\omega)\star$ and $\star \widehat{c}(\omega)=(-1)^{n+1}\widehat{c}(\omega)\star$,
\item $\star\varepsilon=(-1)^{n+1}\varepsilon\star$.
\item $\star D =(-1)^{n}D\star$.
\end{enumerate}
\end{proposition}

The strategy is to study $S(D)$ through the decomposition of Corollary \ref{Coro:DecInvForm}, that is 
\begin{align*}
S(D):=\left(d+(-1)^{n}\star d\star\right)\bigg{|}_{\Omega_c(M_0)^{S^1}}:
\xymatrixrowsep{2cm}\xymatrixcolsep{2cm}\xymatrix{
{
\begin{array}{c}
\Omega_{\text{bas},c}(M_0)^{S^1}\\
\bigoplus\\
\Omega_{\text{bas},c}(M_0)^{S^1}
\end{array}
}\ar[r] &
{
\begin{array}{c}
\Omega_{\text{bas},c}(M_0)^{S^1}\\
\bigoplus\\
\Omega_{\text{bas},c}(M_0)^{S^1}.
\end{array}
}
}
\end{align*}
For the decomposition of $\star$ we follow the techniques of \cite[Chapter 7]{T97}. 
\begin{definition}\label{Def:Bar*}
The {\em basic Hodge star operator} is defined as the linear map
\begin{align*}
\bar{*}:\xymatrixrowsep{0.01cm}\xymatrixcolsep{2cm}\xymatrix{
\Omega^j_{\textnormal{bas}}(M_0)  \ar[r] & \Omega^{n-j}_{\textnormal{bas}} (M_0),
}
\end{align*}
satisfying the conditions
\begin{align}
 \bar{*}\beta=&(-1)^{n-j}*(\beta\wedge\chi), \label{Eqn:TransHosgeStar1} \\
 *\beta=&\bar{*}\beta\wedge\chi, \label{Eqn:TransHosgeStar2}
\end{align}
where $*$ is the Hodge star operator of $M$.
\end{definition} 

\begin{remark}\label{Rmk:VolumeQuot}
Observe that the volume form can be written as $ \vol_{M_0}=* 1=\bar{*}1\wedge\chi$.
\end{remark}

\begin{lemma}
The operator $\bar{*}$ satisfies $\bar{*}^2=(-1)^{j(n-j)}$ on $j$-forms.
\end{lemma}

In view of this lemma we can define a chirality operator on basic differential forms as in \cite[Section 5]{HR13}. The following result follows from Proposition \ref{Prop:Chirl}.

\begin{proposition}\label{Prop:BarChirl}
The basic chirality operator
\begin{align*}
\bar{\star}:\xymatrixrowsep{0.01cm}\xymatrixcolsep{2cm}\xymatrix{
\Omega^j_{\textnormal{bas}}(M_0)  \ar[r] & \Omega^{n-j}_{\textnormal{bas}} (M_0)
}
\end{align*}
defined by $\bar{\star}\coloneqq i^{[(n+1)/2]+2nj+j(j-1)}\bar{*}$, satisfies the relations analogous to Proposition \ref{Prop:Chirl}.
\end{proposition}

\begin{lemma}\label{Lemma:star}
With respect to the decomposition of Corollary \ref{Coro:DecInvForm} we can express the operator $\star$ as 
\begin{equation*}
\star\bigg{|}_{\Omega(M_0)^{S^1}}=i^{q(n)}(-1)^n\left(\begin{array}{cc}
0 & -\varepsilon\bar{\star}\\
\varepsilon\bar{\star} & 0
\end{array}
\right),
\end{equation*}
where $q(n)\coloneqq (n-1)\textnormal{mod}(2)$.
\end{lemma}
\begin{proof}
Recall that $[\cdot]$ denotes the integer part function. First observe the relation
 \begin{align*}
\left[\frac{n+1}{2}\right]+q(n)=\left[\frac{n}{2}\right]+1.
\end{align*}
For $\beta$ a basic $j$-form we calculate,
\begin{align*}
\star\beta =&i^{[n/2]+1+2(n+1)j+j(j-1)}*\beta\\
=&i^{q(n)+2j+[(n+1)/2]+2nj+j(j-1)}\bar{*}\beta\wedge\chi\\
=&(i^{q(n)}\bar{\star}\varepsilon \beta)\wedge\chi\\
=&(i^{q(n)}(-1)^n\varepsilon\bar{\star} \beta)\wedge\chi.
\end{align*}
On the other hand using Proposition \ref{Prop:Chirl} we compute,
\begin{align*}
\star (\beta\wedge\chi)=&(\star\circ(\chi\wedge)\circ \varepsilon )\beta= (-1)^{n+1}(\iota_{\chi^\sharp}\circ \star\circ\varepsilon)\beta=(\iota_{\chi^\sharp}\circ\varepsilon\circ \star)\beta.
\end{align*}
Finally, using the first computation above we conclude that 
\begin{align*}
(\iota_{\chi^\sharp}\circ\varepsilon\circ \star)\beta=&\iota_{\chi^\sharp}\varepsilon((i^{q(n)}(-1)^n\varepsilon\bar{\star} \beta)\wedge\chi)
=-i^{q(n)}(-1)^n\iota_{\chi^\sharp}(\chi\wedge (\varepsilon\bar{\star} \beta))
=-i^{q(n)}(-1)^n\varepsilon\bar{\star} \beta.
\end{align*}
\end{proof}

We are now ready to describe the operator $S(D)$ of Lemma \ref{Lemma:OpS}.
\begin{theorem}\label{Thm:OpInv}
With respect to the decomposition of Corollary \ref{Coro:DecInvForm} the exterior derivative decomposes as
\begin{align*}
d\bigg{|}_{\Omega(M_0)^{S^1}}=\left(\begin{array}{cc}
d & \varepsilon\varphi_0\wedge\\
0 & d-\kappa\wedge
\end{array}
\right)
\end{align*}
and its formal adjoint as
\begin{align*}
d^\dagger\bigg{|}_{\Omega(M_0)^{S^1}}=
\left(\begin{array}{cc}
(-1)^{n+1}\bar{\star}d\bar{\star} +\iota_H & 0\\
-\varepsilon\bar{\star}(\varphi_0\wedge)\bar{\star} & (-1)^{n+1}\bar{\star}d\bar{\star}
\end{array}
\right).
\end{align*}
Hence, the restriction of the Hodge-de Rham operator $D$ to the space of $S^1$-invariant forms with respect to this decomposition is
\begin{equation*}
S(D)\coloneqq
D\bigg{|}_{\Omega(M_0)^{S^1}}=
\left(\begin{array}{cc}
d+(-1)^{n+1}\bar{\star}d\bar{\star} +\iota_H & \varepsilon(\varphi_0\wedge)\\
-\varepsilon\bar{\star}(\varphi_0\wedge)\bar{\star} & d+(-1)^{n+1}\bar{\star}d\bar{\star}-\kappa\wedge
\end{array}
\right).
\end{equation*}
\end{theorem}
\begin{proof}
Let $\omega_0+\omega_1\wedge\chi\in\Omega(M_0)^{S^1}$, using \eqref{Eqn:dchi} we compute (cf. \cite[Proposition 10.1]{BGV}),
\begin{align*}
d(\omega_0+\omega_1\wedge\chi)
=&d\omega_0+d\omega_1\wedge\chi-(\varepsilon\omega_1)\wedge \kappa\wedge\chi+(\varepsilon\omega_1)\wedge\varphi_0\\
=&(d\omega_0+\varepsilon\varphi_0\wedge\omega_1)+(d\omega_1-\kappa\wedge\omega_1)\wedge\chi, 
\end{align*}
from where we obtain the desired decomposition for the exterior derivative. For the adjoint, we first calculate using the decomposition of $d$ and Lemma \ref{Lemma:star},
\begin{align*}
d\star \bigg{|}_{\Omega(M_0)^{S^1}}=&
i^{q(n)}(-1)^n
\left(\begin{array}{cc}
d & \varepsilon\varphi_0\wedge\\
0 & d-\kappa\wedge
\end{array}
\right)
\left(\begin{array}{cc}
0 & -\varepsilon\bar{\star}\\
\varepsilon \bar{\star}& 0 
\end{array}
\right)\\
=& 
i^{q(n)}(-1)^n\left(\begin{array}{cc}
\varphi_0\wedge\bar{\star}& -d\varepsilon \bar{\star}\\
(d-\kappa\wedge)\varepsilon\bar{\star} & 0
\end{array}
\right)\\
=&
i^{q(n)}(-1)^n\left(\begin{array}{cc}
\varphi_0\wedge\bar{\star}& \varepsilon d\bar{\star}\\
-\varepsilon(d-\kappa\wedge)\bar{\star} & 0
\end{array}
\right).
\end{align*}

The result then follows from the relation $(i^{q(n)})^2=(-1)^{n+1}$ and  Proposition \ref{Prop:BarChirl}.
\end{proof}

\subsubsection{Construction of the operator $T(D)$}

Now that we have described the isomorphism $\Phi$ and the operator $S(D)$ we we can compute the self-adjoint operator $T\coloneqq\Phi\circ S\circ \Phi^{-1}$ of Proposition \ref{Prop:OpT}. As $D$ is a first order differential operator, Proposition \ref{Prop:SDiff} ensures $T$ is also generated by a differential operator of the same order.  Let us begin with the  Hodge star operator. In view of Remark \ref{Rmk:VolumeQuot}, we choose the sign of  volume form $\vol_{M_0/S^1}$ on $M_0/S^1$ so that $\pi_{S^1}^*(\vol_{M_0/S^1})\coloneqq\bar{*}1$. This means that we can express $\vol_{M_0}=\pi_{S^1}^*(\vol_{M_0/S^1})\wedge\chi$. With this choice we can identify $\bar{*}$, via the orbit map $\pi_{S^1}$, with the Hodge star operator $*_{M_0/S^1}$ of $M_0/S^1$ with respect to the quotient metric. Moreover, the following diagram commutes \cite[Corollary 4.25]{JO17}
\begin{align}\label{Diag:StarBar}
\xymatrixrowsep{2cm}\xymatrixcolsep{2cm}\xymatrix{
\Omega_{\textnormal{bas}}(M_0)  \ar[r]^-{\bar{\star}} & \Omega_{\textnormal{bas}} (M_0)\\
\Omega(M_0/S^1) \ar[u]^-{\pi_{S^1}^*} \ar[r]^-{\star_{M_0/S^1}}& \Omega(M_0/S^1).\ar[u]_-{\pi_{S^1}^*}
}
\end{align}

Now let us study the zero order terms. Since the forms $\kappa$ and $\varphi_0$ are both basic then there exist unique $\bar{\kappa}\in\Omega^1(M_0/S^1)$ and $\bar{\varphi}_0\in\Omega^2(M_0/S^1)$ such that $\kappa=\pi^*_{S^1}(\bar{\kappa})$ and $\varphi_0=\pi^*_{S^1}(\bar{\varphi}_0)$. Moreover, as pullback commutes with the wedge product, then the following diagram commute (similarly for $\varphi_0\wedge$)
\begin{align}\label{Diag:KappaAction}
\xymatrixrowsep{2cm}\xymatrixcolsep{2cm}\xymatrix{
\Omega_{\textnormal{bas}}(M_0)  \ar[r]^-{\kappa\wedge} & \Omega_{\textnormal{bas}} (M_0)\\
\Omega(M_0/S^1) \ar[u]^-{\pi_{S^1}^*} \ar[r]^-{\bar{\kappa}\wedge}& \Omega(M_0/S^1), \ar[u]_-{\pi_{S^1}^*}
}
\end{align}

Next we consider the term $-\bar{\star}({\varphi}_0\wedge)\bar{\star}$ as an operator on the quotient space. In view of  \eqref{Diag:StarBar} and to lighten the notation we are going to identify $\bar{\star}\equiv\star_{M_0/S^1}$.
\begin{proposition}[{\cite[Proposition 4.26]{JO17}}]\label{Prop:CliffMultVarphi}
With respect to the quotient metric on $M_0/S^1$ we have $(\bar{\varphi}_0\wedge)^\dagger=-\bar{\star}(\bar{\varphi_0}\wedge)\bar{\star}$.
In particular, the operator $\widehat{c}(\bar{\varphi}_0)\coloneqq(\bar{\varphi}_0\wedge)+(\bar{\varphi}_0\wedge)^\dagger$ satisfies 
\begin{enumerate}
\item $\widehat{c}(\bar{\varphi}_0)\bar{\star}+\bar{\star}\widehat{c}(\bar{\varphi}_0)=0$.
\item $\widehat{c}(\bar{\varphi}_0)\varepsilon-\varepsilon\widehat{c}(\bar{\varphi}_0)=0$. 
\end{enumerate}
\end{proposition}

Finally we treat the first order terms of $S(D)$ in Theorem \ref{Thm:OpInv}.  For the exterior derivative, as it also commutes with pullbacks, we have an analogous commutative diagram, 
\begin{align}\label{Diag:dBasic}
\xymatrixrowsep{2cm}\xymatrixcolsep{2cm}\xymatrix{
\Omega_{\textnormal{bas}}(M_0)  \ar[r]^-{d} & \Omega_{\textnormal{bas}} (M_0)\\
\Omega(M_0/S^1) \ar[u]^-{\pi_{S^1}^*} \ar[r]^-{d_{M_0/S^1}}& \Omega(M_0/S^1), \ar[u]_-{\pi_{S^1}^*}
}
\end{align}
where $d_{M_0/S^1}$ is the exterior derivative of $M_0/S^1$. Hence, it remains to study the operator 
\begin{align*}
(-1)^{n+1}\bar{\star}d\bar{\star}:
\xymatrixrowsep{2cm}\xymatrixcolsep{2cm}\xymatrix{
\Omega^j_{\textnormal{bas}}(M_0)  \ar[r] & \Omega^{j-1}_{\textnormal{bas}} (M_0).
}
\end{align*}
\begin{remark}\label{Rmk:AdjointNotPresBas}
Let $d^\dagger_{M_0/S^1}=(-1)^{n+1}\star_{M_0/S^1}d_{M_0/S^1}\star_{M_0/S^1}$ be the $L^2$-formal adjoint of $d_{M_0/S^1}$ with respect to the quotient metric (Proposition \ref{Prop:Chirl}(4)). One might think that there is an analogous commutative diagram as \eqref{Diag:dBasic} where $d^{\dagger}$ and $d^\dagger_{M_0/S^1}$ are placed instead. Note however that $d^\dagger$ does not preserve the space of basic forms, as it can be explicitly seen from Theorem \ref{Thm:OpInv}, and in general
$d^\dagger\circ \pi_{S^1}^* \neq\pi_{S^1}^*\circ d^\dagger_{M_0/S^1}.$
\end{remark}

Observe that \eqref{Diag:StarBar} and  \eqref{Diag:dBasic} can be combined to obtain the following commutative diagram 
\begin{align}\label{Diag:ddaggerBasic}
\xymatrixrowsep{2cm}\xymatrixcolsep{2cm}\xymatrix{
\Omega_{\textnormal{bas}}(M_0)  \ar[r]^-{(-1)^{n+1}\bar{\star}d\bar{\star}} & \Omega_{\textnormal{bas}} (M_0)\\
\Omega(M_0/S^1) \ar[u]^-{\pi_{S^1}^*} \ar[r]^-{d^{\dagger}_{M_0/S^1}}& \Omega(M_0/S^1). \ar[u]_-{\pi_{S^1}^*}
}
\end{align}
Altogether, from the discussion of Section \ref{Section:ConstrF}, Theorem \ref{Thm:OpInv}, \eqref{Diag:StarBar}, \eqref{Diag:KappaAction}, \eqref{Diag:dBasic} and \eqref{Diag:ddaggerBasic} we can describe explicitly  the operator $T(D)$ of Proposition \ref{Diag:dBasic}.

\begin{theorem}\label{Theorem:OpT}
The operator $T(D)$ of Proposition \ref{Prop:OpT} for a semi-free $S^1$-action can be written as
\begin{equation*}
T(D)=\left(
\begin{array}{cc}
D_{M_0/S^1}+\iota_{\bar{\kappa}^\sharp} & \varepsilon(\bar{\varphi_0}\wedge) \\
 \varepsilon (\bar{\varphi}_0\wedge)^\dagger & D_{M_0/S^1}-\bar{\kappa}\wedge
\end{array}
\right),
\end{equation*}
where $D_{M_0/S^1}\coloneqq d_{M_0/S^1}+d^\dagger_{M_0/S^1}$is the Hodge-de Rham operator on $M_0/S^1$. The operator $T(D)$, when defined on the core $\Omega_c(M/S^1)$, is essentially self-adjoint. 
\end{theorem}

\subsection{Dirac-Schr\"odinger operators}

The operator $T\coloneqq T(D)$ of Theorem \ref{Theorem:OpT} is self-adjoint in $L^2(F,h)$ but not in $L^{2}(F)$, the $L^2$-inner product without the weight $h$. For example, the adjoint in $L^2(F)$ of $D_{M_0/S^1}+\iota_{\bar{H}}$ is $(D_{M_0/S^1}+\iota_{\bar{H}})^\dagger=D_{M_0/S^1}+\bar{\kappa}\wedge$. To obtain a self-adjoint operator in $L^2(F)$ we perform the following unitary transformation:
\begin{equation}\label{Eqn:U}
\omega=\left(
\begin{array}{c}
\omega_0\\
\omega_1
\end{array}
\right)\longmapsto 
U(\omega)\coloneqq
h^{-1/2}
\left(
\begin{array}{c}
\omega_0\\
\omega_1
\end{array}
\right),
\end{equation}
for $\omega_0,\omega_1\in\Omega_c(M_0/S^1)$. Note that $\norm{U(\omega)}_{L^2(F,h)}=\norm{\omega}_{L^2(F)}$. Using this transformation we want to compute an explicit formula for the operator $\widehat{T}\coloneqq U^{-1}TU$, defined on  $\text{Dom}(\widehat{T} )\coloneqq U^{-1}(\text{Dom}(T))$. 
\begin{lemma}\label{Lemmadh}
The volume of the orbit function $h:M_0/S^1\longrightarrow \mathbb{R}$ satisfies 
\begin{align*}
d(h^{\pm1/2})=\mp\frac{1}{2}h^{\pm1/2}\bar{\kappa}.
\end{align*}
\end{lemma}
\begin{proof}
It follows directly from Lemma \ref{Lemma:dh}.
\end{proof}
The transformation formula follows directly from this lemma and Theorem \ref{Theorem:OpT}.
\begin{theorem}[{\cite[Theorem 4.31]{JO17}}]\label{Thm:THat}
The operator $\widehat{T}$ is given by 
\begin{equation*}
\widehat{T}=
\left(
\begin{array}{cc}
D_{M_0/S^1}+\frac{1}{2}\widehat{c}(\bar{\kappa})&\varepsilon \bar{\varphi}_0\wedge\\
\varepsilon(\bar{\varphi}_0\wedge)^\dagger &  D_{M_0/S^1}-\frac{1}{2}\widehat{c}(\bar{\kappa})
\end{array}
\right),
\end{equation*}
where $\widehat{c}(\bar{\kappa})\coloneqq\bar{\kappa}\wedge+\iota_{\bar{\kappa}^\sharp}$ is the right Clifford multiplication by the mean curvature form. 
\end{theorem}

\subsubsection{An involution on $F$}\label{Sec:AnInvolution}

As we are interested in Fredholm indices, we would like to find a self-adjoint involution which anti-commutes with $\widehat{T}$ in order to split this operator. Since the dimension of $M_0/S^1$ is $n$ then $\bar{\star}D_{M_0/S^1}+(-1)^n D_{M_0/S^1}\bar{\star}=0$, thus a  first natural candidate is 
\begin{align}\label{Def:BigStar}
{\bigstar}\coloneqq
\left(
\begin{array}{cc}
0 & \bar{\star}\\
\bar{\star} &0
\end{array}
\right).
\end{align}
From Proposition \ref{Prop:Chirl} and Proposition \ref{Prop:CliffMultVarphi} we verify $\bigstar\widehat{T} =(-1)^{n+1}\widehat{T}\bigstar$. This implies that if $n$ is even then we can decompose 
\begin{align*}
\widehat{T}=
\left(
\begin{array}{cc}
0 & \widehat{T}^-\\
\widehat{T}^+ & 0
\end{array}
\right),
\end{align*}
with respect to the involution $\bigstar$. Nevertheless, by studying the trivial case of spinning  a closed manifold it turns out that the index of $\widehat{T}$  is always zero (\cite[Example 4.33]{JO17}).

\begin{remark}[Induced Dirac-Schr\"odinger Geometric Operators]
In \cite[Section 4.3]{JO17}, two other natural geometric operators on $M$ were pushed down to $M_0/S^1$ following the Br\"uning-Heintze approach: the {\em positive signature operator} and the {\em odd signature operator}. From the construction itself we know the resulting operators are elliptic and self-adjoint. Nevertheless, the induced potential (zero order term) in both cases does not commute with the involution $\bar{\star}$.
\end{remark}
\subsubsection{The Dirac-Schr\"odinger signature operator}\label{Sect:ConstDiracSchrOp}

In view of Proposition \ref{Coro:DecInvForm} and \eqref{Eqn:U} we introduce, for $j=0,1,\cdots 4k$, the unitary transformation
\begin{align*}
\psi_j:
\xymatrixcolsep{2cm}\xymatrixrowsep{0.01cm}\xymatrix{
\Omega_c^{j-1}(M_0/S^1)\oplus \Omega_c^j(M_0/S^1) \ar[r]& \Omega_c^j(M_0)^{S^1}\\
(\omega_{j-1},\omega_{j})\ar@{|->}[r] & h^{-1/2}\left(\pi_{S^1}^*\omega_j+(\pi_{S^1}^*\omega_{j-1})\wedge \chi\right).
}
\end{align*}
The following expressions follow immediately from Lemma \ref{Lemmadh} and Theorem \ref{Thm:THat}.
\begin{lemma}\label{Lemma:Psir}
For the maps $\psi_j$ we have the relations
\begin{align*}
d\psi_j(\omega_{j-1},\omega_j)
=&\psi_{j+1}\left(d_{M_0/S^1}\omega_{j-1}-\frac{1}{2}\bar{\kappa}\wedge\omega_{j-1},d_{M_0/S^1}\omega_{j}+\frac{1}{2}\bar{\kappa}\wedge\omega_j+\varepsilon(\bar{\varphi_0}\wedge)\omega_{j-1}\right),\\
d^\dagger\psi_j(\omega_{j-1},\omega_j)
=&\psi_{j-1}\left(d^{\dagger}_{M_0/S^1}\omega_{j-1}-\frac{1}{2}\iota_{\bar{\kappa}^\sharp} \omega_{j-1}+\varepsilon(\bar{\varphi}_0\wedge)^\dagger \omega_j ,d^{\dagger}_{M_0/S^1}\omega_{j}+\frac{1}{2}\iota_{\bar{\kappa}^\sharp}\omega_j\right).
\end{align*}
\end{lemma}
Now consider the transformations introduced in \cite[Section 5]{BS88},
\begin{align*}
\psi_\text{ev}:&
\xymatrixcolsep{3pc}\xymatrixrowsep{0.5pc}\xymatrix{
\Omega_c(M_0/S^1)\ar[r]& \Omega_c^\text{ev}(M_0)^{S^1}\\
(\omega_0,\cdots,\omega_{4k} )\ar@{|->}[r] & (\psi_0(0,\omega_0),\psi_2(\omega_1,\omega_2),\cdots,\psi_{4k}(\omega_{4k-1},\omega_{4k})),
}\\\\\
\psi_\text{odd}:&
\xymatrixcolsep{3pc}\xymatrixrowsep{0.5pc}\xymatrix{
\Omega_c(M_0/S^1)\ar[r]& \Omega_c^\text{odd}(M_0)^{S^1}\\
(\omega_0,\cdots,\omega_{4k} )\ar@{|->}[r] & (\psi_1(\omega_0,\omega_1),\psi_3(\omega_2,\omega_3),\cdots,\psi_{4k-1}(\omega_{4k-2},\omega_{4k-1})).
}
\end{align*}
Motivated by Lemma \ref{Lemma:Psir} we define the operator
\begin{equation}\label{Eqn:ConstrD'}
\mathscr{D}'\coloneqq \psi_{\text{odd}}^{-1}d\psi_{\text{ev}}+\psi_{\text{ev}}^{-1}d^\dagger\psi_{\text{odd}}:\Omega_c(M_0/S^1)\longrightarrow\Omega_c(M_0/S^1).
\end{equation}
Clearly $\mathscr{D}'$ is symmetric since 
$(\psi_{\text{odd}}^{-1}d \psi_{\text{ev}})^\dagger = \psi_{\text{ev}}^{-1}d^\dagger\psi_{\text{odd}}$. A simple computation using Theorem \ref{Thm:OpInv} allows us to compute the explicit form of $\mathscr{D}'$.

\begin{proposition}[{\cite[Proposition 4.39]{JO17}}]
The operator $\mathscr{D}'$  defined in \eqref{Eqn:ConstrD'} is explicitly given by 
\begin{equation*}
\mathscr{D}'=D_{M_0/S^1}+\frac{1}{2}c(\bar{\kappa})\varepsilon-\widehat{c}(\bar{\varphi}_0)\left(\frac{1-\varepsilon}{2}\right).
\end{equation*}
Thus, it is a first order elliptic operator. Moreover, it is symmetric on $\Omega_c(M_0/S^1)$ and anti-commutes with the chirality operator $\bar{\star}$. 
\end{proposition}
\begin{proof}
Observe for the principal symbol that $\sigma_P(\mathscr{D}')=\sigma_P(D_{M_0/S^1})$, so indeed $\mathscr{D}'$ is first order elliptic. For the symmetry assertion we verify
$(c(\bar{\kappa})\varepsilon)^\dagger=\varepsilon c(\bar{\kappa})^{\dagger}=-\varepsilon c(\bar{\kappa})=c(\bar{\kappa})\varepsilon,$
and 
$(\widehat{c}(\bar{\varphi}_0)(1-\varepsilon))^\dagger=(1-\varepsilon)\widehat{c}(\bar{\varphi}_0)=\widehat{c}(\bar{\varphi}_0)(1-\varepsilon)$. For the last assertion we calculate using Proposition \ref{Prop:CliffMultVarphi}, 
\begin{align*}
\bar{\star}\mathscr{D}'
=&-D_{M_0/S^1}\bar{\star}-\frac{1}{2}c(\bar{\kappa})\bar{\star}\varepsilon+\widehat{c}(\bar{\varphi}_0)\bar{\star}\left(\frac{1-\varepsilon}{2}\right)\\
=&-D_{M_0/S^1}\bar{\star}-\frac{1}{2}c(\bar{\kappa})\varepsilon\bar{\star}+\widehat{c}(\bar{\varphi}_0)\left(\frac{1-\varepsilon}{2}\right)\bar{\star}.
\end{align*}
\end{proof}

Due the nature of the potential of the operator $\mathscr{D}'$ and in view of Theorem \ref{Thm:THat} we would expect $\mathscr{D}'$ to be essentially self-adjoint on the core $\Omega_c(M_0/S^1)$. In order to apply the construction of Br\"uning and Heintze we need to find an operator on $M$, commuting with the $S^1$-action, so that when pushed down to $M_0/S^1$ coincides with $\mathscr{D}'$. Let us explore how to find such an operator. In view of \eqref{Eqn:ConstrD'} we define $\text{d}\coloneqq \psi_\text{odd}^{-1}d\psi_\text{ev}$ and $\text{d}^\dagger\coloneqq \psi_\text{ev}^{-1}d^\dagger\psi_\text{odd}$ so that $\mathscr{D}'=\text{d}+\text{d}^\dagger$.  Observe that these operators fit in the commutative diagrams
\begin{align*}
\xymatrixcolsep{4pc}\xymatrixrowsep{3pc}\xymatrix{
\Omega_c^\text{ev}(M_0)^{S^1} \ar[r]^-{d} & \Omega_c^\text{odd}(M_0)^{S^1} \ar[r]^-{\psi_\text{ev}\psi_\text{odd}^{-1}} & \Omega_c^\text{ev}(M_0)^{S^1}\\
\Omega_c(M_0/S^1) \ar[u]^-{\psi_\text{ev}} \ar[r]^{\text{d}} & \Omega_c(M_0/S^1) \ar[u]^-{\psi_\text{odd}} \ar@{=}[r] & \Omega_c(M_0/S^1), \ar[u]^-{\psi_\text{ev}} 
}
\end{align*}
\begin{align*}
\xymatrixcolsep{4pc}\xymatrixrowsep{3pc}\xymatrix{
\Omega_c^\text{ev}(M_0)^{S^1} \ar[r]^-{\psi_\text{odd}\psi_\text{ev}^{-1}} & \Omega_c^\text{odd}(M_0)^{S^1} \ar[r]^-{d^\dagger} & \Omega_c^\text{ev}(M_0)^{S^1}\\
\Omega_c(M_0/S^1) \ar[u]^-{\psi_\text{ev}} \ar@{=}[r]& \Omega_c(M_0/S^1) \ar[u]^-{\psi_\text{odd}} \ar[r]^{\text{d}^\dagger} & \Omega_c(M_0/S^1).\ar[u]^-{\psi_\text{ev}} 
}
\end{align*}
\begin{lemma}
Let $c(\chi)$ be the left Clifford multiplication by the characteristic $1$-form. Then $\psi_\textnormal{ev}\psi_\textnormal{odd}^{-1}=-c(\chi),$ and $\psi_\textnormal{odd}\psi_\textnormal{ev}^{-1}= c(\chi)$. 
\end{lemma}
\begin{proof}
This follows directly from the definition of $\psi_\textnormal{ev}/\psi_\textnormal{odd}$.
\end{proof}
From this lemma we obtain the commutative diagram
\begin{align}\label{Diag:B}
\xymatrixcolsep{2cm}\xymatrixrowsep{2cm}\xymatrix{
\Omega_c^\text{ev}(M_0)^{S^1} \ar[r]^-{-c(\chi)d+d^\dagger c(\chi)} & \Omega_c^\text{ev}(M_0)^{S^1} \\
\Omega_c(M_0/S^1) \ar[u]^-{\psi_\text{ev}} \ar[r]^-{\mathscr{D}'}& \Omega_c(M_0/S^1), \ar[u]_-{\psi_\text{ev}}
 }
\end{align}
and conclude that the operator $\mathscr{D}'$ is unitary equivalent to the operator $ -c(\chi)d+c(\chi)d^\dagger$ when restricted to $\Omega_c^\text{ev}(M_0)^{S^1} $.
\begin{proposition}\label{Prop:OpB}
The operator $B\coloneqq -c(\chi)d+d^\dagger c(\chi)$ satisfies:
\begin{enumerate}
\item It is a transversally elliptic first order differential operator with principal symbol
\begin{align*}
\sigma_P(B)(x,\xi)=-i(\inner{\chi}{\xi}+c(\xi)c(\chi)).
\end{align*}
\item It can be extended to $M$ and $B:\Omega(M)\longrightarrow \Omega(M)$  it essentially self-adjoint when defined on this core. 
\item It commutes with the $S^1$-action on differential forms.
\item It commutes with the Gau\ss-Bonnet involution $\varepsilon$ and therefore it can be decomposed as $B=B^\textnormal{ev}\oplus B^\textnormal{odd}$ where $B^\textnormal{ev/odd}:\Omega^{\textnormal{ev/odd}}(M)\longrightarrow \Omega^{\textnormal{ev/odd}}(M)$. 
\end{enumerate}
\end{proposition}
\begin{proof}
To prove the first statement recall the expressions for the principal symbols
\begin{align*}
\sigma_P(d)(x,\xi)=&-i\xi\wedge,\\
\sigma_P(d^\dagger)(x,\xi)=&i\iota_{\xi^\sharp}.
\end{align*}
Using the relation
\begin{align*}
c(\chi)\circ (\xi\wedge)=&\chi\wedge\xi\wedge-\iota_{X}\circ (\xi\wedge)=\chi\wedge\xi\wedge -\inner{\chi}{\xi}+\xi\wedge \iota_X =-\xi\wedge c(\chi)-\inner{\chi}{\xi},
\end{align*}
we calculate the principal symbol of the operator $B$,
\begin{align*}
\sigma_P(B)(x,\xi)
=&i\left(c(\chi)\circ(\xi\wedge)  +(\iota_{\xi^\sharp})\circ c(\chi)\right)\\
=&i\left(-\xi\wedge c(\chi)-\inner{\chi}{\xi} +(\iota_{\xi^\sharp})\circ c(\chi)\right)\\
=&-i(\inner{\chi}{\xi}+c(\xi)c(\chi)).
\end{align*}
In particular we see that if $\inner{\chi}{\xi}=0$ then $\sigma_P(B)(x,\xi)^2=\norm{\xi}^2$, so we see that $B$ is a transversally elliptic first order differential operator. To prove the second assertion observe that the Clifford multiplication operator $c(\chi)$ has domain $\dom(c(\chi))=\Omega_c(M_0)$. For $\omega\in\Omega_c(M_0)$ we have $\norm{c(\chi)\omega}^2_{L^2(\wedge_\mathbb{B} T^* M)} = \norm{\omega}^2_{L^2(\wedge_\mathbb{C} T^* M)}$ see we can extend the operator $c(\chi)$ to all $L^2(M,\wedge_\mathbb{C} T^*M)$ by density. Using this fact we see that $B$ is indeed densely defined with core $\Omega(M)$.
Since $M$ is compact we can use Remark \cite[Lemma 2.1]{GL02} to conclude that $B$ is an essentially self-adjoint operator. The last two assertions follow easily from Proposition \ref{Prop:S1commD} and the fact that $\varepsilon$ anti-commutes with the left Clifford action.
\end{proof}

Now we can implement the construction described in Section \ref{Section:InduedOperatorsGen} in this setting: the restriction of $B$ to the $S^1$-invariant forms remains essentially self-adjoint and since this operator is unitary equivalent to $\mathscr{D}'$ through $\psi_\text{ev}$ we conclude that $\mathscr{D}'$ is essentially self-adjoint with core $\Omega_c(M_0/S^1)$. We summarize these results in the next theorem.

\begin{theorem}\label{Thm:InducedDiracOp}
The Dirac-Schr\"odinger operator $\mathscr{D}'$ defined on on $\Omega_c(M_0/S^1)$, is a first order elliptic differential operator which is essentially self-adjoint.  As the dimension of  $M$ is odd then $\mathscr{D}'$ anti-commutes with the chirality operator $\bar{\star}$ on $M_0/S^1$ and therefore we can define the operator $\mathscr{D}'^{+}:\Omega^+_c(M_0/S^1)\longrightarrow \Omega^-_c(M_0/S^1)$, where $\Omega^\pm_c(M_0/S^1)$ is the  $\pm 1$-eigenspace of $\bar{\star}$. 
\end{theorem}

\begin{remark}\label{Rmk:OpD}
We will later give a detailed description of the operator $\mathscr{D}'$ close to a connected component of the fixed point set. We will see that the term containing $\widehat{c}(\bar{\varphi}_0)$ is actually bounded and therefore, by the Kato-Rellich theorem, it will be enough to consider the operator
\begin{align*}
\mathscr{D}\coloneqq D_{M_0/S^1}+\frac{1}{2}c(\bar{\kappa})\varepsilon.
\end{align*}
We will also see that the factor $1/2$ is fundamental for the essential self-adjointness of $\mathscr{D}$. 
\end{remark}

\begin{remark}[{\cite[Lemma 4.48]{JO17}}]
It is easy to verify that $\mathscr{D}^2$ is a generalized Laplacian in the sense of \cite[Definition 2.2]{BGV}. 
\end{remark}

Even though we are not going to make use of it, we want to finish this section by showing that the operator $\mathscr{D}'$ is also essentially self-adjoint  whenever $M$ is a complete (not necessarily compact) manifold on which $S^1$ acts by orientation preserving isometries. The strategy of the proof is inspired in the similar result for Dirac operators. 

\begin{lemma}[{\cite[Lemma 4.50]{JO17}}]\label{lemma:CommBf}
For $f\in C^\infty(M)$ we have $[B,f]=c(df)c(\chi)+\inner{\chi}{d f}$, where $[\cdot, \cdot]$ denotes the commutator. 
\end{lemma}
\begin{proof}
We use Proposition \ref{Prop:Chirl}(3) to compute as in the proof of Proposition \ref{Prop:OpB}.
\end{proof}

\begin{coro}[{\cite[Corollary 4.51]{JO17}}]\label{Coro:BComplete}
Let $M$ be a complete Riemannian manifold. Then the operator $B$ is essentially self-adjoint. 
\end{coro}

\begin{proof}
If follows from \ref{lemma:CommBf} analogously as for the proof for Dirac operators given in \cite[Chapter 4]{F00}, \cite[Theorem II.5.7]{LM89} and \cite{W73}.
\end{proof}

\begin{remark}[The Basic Signature Operator]
Motivation for the structure of the operator $\mathscr{D}'$ of Theorem \ref{Thm:InducedDiracOp} comes from the work of Habib and Richardson on modified differentials in the context of Riemannian foliations \cite{HR13}. In their setting they defined the {\em basic signature operator} action on basic forms. In \cite[Section 4.4]{JO17} a detailed comparison between this operator and $\mathscr{D}'$ was made. It was shown that, when pushing down to the quotient space, the basic signature operator can be identified with $D_{M_0/S^1}$ and the twist of the de Rham differential can be seen as a consequence of the transformation  \eqref{Eqn:U}. 
\end{remark}

\begin{example}[The 2-sphere]
We consider the semi-free $S^1$-action on the unit $2$-sphere $M=S^2\subset\mathbb{R}^3$ by rotations along the $z$-axis. The fixed point set is $M^{S^1}=\{\mathcal{N},\mathcal{S}\}$, where $\mathcal{N}$ and $\mathcal{S}$ denote the north and south pole respectively. On the complement $M_0=S^2-\{\mathcal{N},\mathcal{S}\}$ the action is free. We equip $S^2$ with the induced metric coming from the Euclidean inner product of $\mathbb{R}^3$ which can be written in polar coordinates as
\begin{equation}\label{Eq:metricS2}
g^{TS^2}=d\theta^2+\sin^2\theta d\phi^2.
\end{equation}
The quotient manifold $M_0/S^1$ can be identified with the open interval $I\coloneqq (0,\pi)$, which we equip with the flat metric $g^{TI}=d\theta^2$ so that the orbit map $\pi_{S^1}:M_0\longrightarrow M_0/S^1$ becomes a Riemannian submersion. The generating vector field of the action is clearly $V=\partial_\phi$, the characteristic $1$-from is $\chi=\sin\theta d\phi$ and the corresponding mean curvature form is $\kappa=-\cot\theta d\theta$. Using \eqref{Eqn:dchi} we find that $\varphi_0=0$. As a result, the operator $\mathscr{D}'$ of Theorem \ref{Thm:InducedDiracOp}
is 
\begin{align*}
\mathscr{D}'=\mathscr{D}=D_I+\frac{1}{2}\cot\theta c(d\theta)\varepsilon.
\end{align*}
With respect to the degree decomposition, we can express it as 
\begin{equation*}\label{Eqn:D0S2}
\mathscr{D}=\gamma\left(\frac{\partial}{\partial \theta}+
\left(
\begin{array}{cc}
-1/2 & 0\\
0 & 1/2
\end{array}
\right)
\otimes \cot\theta\right),\quad\text{where}\quad
\gamma\coloneqq
\left(
\begin{array}{cc}
0 & -1\\
1 & 0
\end{array}
\right).
\end{equation*}
When $\theta\longrightarrow 0$ this operator takes the form
\begin{align*}
\gamma\left(\frac{\partial}{\partial \theta}+
\left(
\begin{array}{cc}
-1/2 & 0\\
0 & 1/2
\end{array}
\right)
\otimes \frac{1}{\theta}\right).
\end{align*} 
In view of \cite[Theorem 3.2]{BS88} we see that this operator has the structure of a {\em first order regular singular operator}, in the sense of Br\"uning and Seeley, and that indeed  is essentially self-adjoint on the core $\Omega_c(I)$.
\end{example}

\begin{remark}
We refer to \cite[Section 5]{JO17} to see more explicit examples of the construction of the operator $\mathscr{D}'$ in higher dimensions. 
\end{remark}

\section{Local description of $\mathscr{D}$ close to $M^{S^1}$}\label{Sect:LocalDesc}

\subsection{Mean curvature $1$-form on the normal bundle}
Recall from Section \ref{Sect:PfoofSignatureFormula} that the strategy to proof the equivariant $S^1$-signature formula of  \cite[Theorem 4]{L00} is to decompose the quotient space as $M_0/S^1=Z_t\cup U_t$ where $Z_t\coloneqq M_0/S^1-N_t(F)$ is a compact manifold with boundary and $U_t\coloneqq N_t(F)$ is the $t$-neighborhood of $F$ in $M/S^1,$ as schematically visualized in Figure \ref{Fig:DecompQuotSpace}. One models $U_t$ as the mapping cylinder of a Riemannian fibration $\pi_\mathcal{F}:\mathcal{F}\longrightarrow F$ , with fiber $\mathbb{C}P^N$, whose total space $\mathcal{F}$ comes as the quotient space $\mathcal{S}/S^1$ of the sphere bundle $\mathcal{S}\coloneqq SNF$ of the normal bundle of $F$ in $M$. 

\begin{figure}[H]
\begin{center}
\begin{tikzpicture}
\draw[rounded corners=32pt](7,-1)--(4,-1)--(2,-2)--(0,0) -- (2,2)--(4,1)--(7,1);
\draw (1.5,0.2) arc (175:315:1cm and 0.5cm);
\draw (3,-0.28) arc (-30:180:0.7cm and 0.3cm);
\draw (7.5,0) arc (0:360:0.5cm and 1cm);
\node (a) at (20:2.5) {$Z_t$};
\node (a) at (7,-1.5) {$\partial Z_t=\mathcal{F}_{t}$};
\draw (11,-0.6)--(11,0.6);
\draw (7,1)--(11,0.6);
\draw (7,-1)--(11,-0.6);
\node (a) at (9,0) {$U_t$};
\node (a) at (11.5,0) {$M^{S^1}$};
\end{tikzpicture}
\caption{Decomposition of $M/S^1$ as a manifold with boundary and a $t$-neighborhood of the fixed point set. }\label{Fig:DecompQuotSpace}
\end{center}
\end{figure}
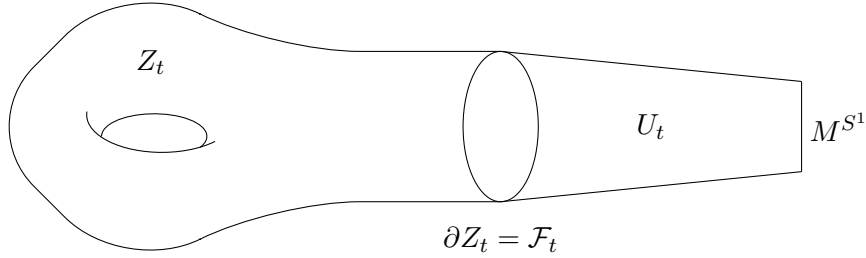

Let $\pi_\mathcal{S}:\mathcal{S}\longrightarrow F$  denote the projection and consider the decomposition of the tangent bundle $T\mathcal{S}=T_V\mathcal{S}\oplus T_H\mathcal{S}$ on which the metric decomposes as $g^{T\mathcal{S}}=g^{T_V\mathcal{S}}\oplus g ^{T_H\mathcal{S}}$, where we identify $T_H\mathcal{S}\cong TF$ via $\pi_\mathcal{S}$. We choose a local oriented orthorormal basis for $T\mathcal{S}$ of the form 
\begin{align}\label{Eqn:ONBTSN}
\{{\underline{e}}_i\}_{i=0}^{v}\cup\{{\underline{f}}_\alpha\}_{\alpha=1}^{h},
\end{align} 
where  $v\coloneqq 2N $ and $h\coloneqq 4k-2N-1$.  Here $\underline{e}_i$ and $\underline{f}_\alpha$ are vertical and horizontal vector fields respectively. Let $\{{\underline{e}}^i\}_{i=0}^{v}\cup\{{\underline{f}}^\alpha\}_{\alpha=1}^{h}$
denote the associated dual basis. In view of Proposition \cite[Lemma 2.2]{U70} we can assume with out loss of generality that the generating vector field of the free $S^1$-action on $\mathcal{S}$ is $V_\mathcal{S}\coloneqq{\underline{e}}_0\in C^\infty(\mathcal{S}, T_V \mathcal{S})$. This implies the corresponding mean curvature form $\kappa_\mathcal{S}$ vanishes by Proposition \ref{Prop:NormV}(2) since $\norm{{\underline{e}}_0}=1$.

\begin{remark}\label{Rmk:BasisBasicForms}
We can assume that the  $1$-forms $\{{\underline{e}}^i\}_{i=1}^{2N}\cup\{{\underline{f}}^\alpha\}_{\alpha=1}^{4k-2N-1}$ are basic.
\end{remark}

Let $\widehat{\underline{\omega}}$ be the connection $1$-form of the Levi-Civita connection of the metric $g^{T\mathcal{S}}$ associated to the orthonormal frame above. Recall that its components satisfy the structure equations
\begin{align*}
d{\underline{e}}^i+{\underline{\omega}}^i_j\wedge {\underline{e}}^j+{\underline{\omega}}^i_\alpha\wedge {\underline{f}}^\alpha=0,\\
d{\underline{f}}^\alpha+{\underline{\omega}}^\alpha_j\wedge {\underline{e}}^j+{\underline{\omega}}^\alpha_\beta\wedge {\underline{f}}^\beta=0.
\end{align*}
Since the characteristic $1$-form is $\chi_{\mathcal{S}}={\underline{e}}^0$ and $\kappa_\mathcal{S}=0$ we can use these structure equations to compute the $2$-form $\varphi_{0,\mathcal{S}}$  from \eqref{Eqn:dchi}, namely
\begin{equation}\label{Eqn:Varphi0S}
\varphi_{0,\mathcal{S}}=d\chi_{\mathcal{S}}=d\underline{e}^0=-{\underline{\omega}}^0_i\wedge {\underline{e}}^i-{\underline{\omega}}^0_\alpha\wedge {\underline{f}}^\alpha.
\end{equation}

Let $\mathring{\mathcal{D}}\coloneqq \mathcal{D}-F$ denote the disk bundle of the normal bundle without the zero section, where the induced $S^1$-action is still free. Now we calculate the corresponding mean curvature $1$-form $\kappa_{\mathring{\mathcal{D}}}$ and the $2$-form $\varphi_{0,\mathring{\mathcal{D}}}$ for the metric  
\begin{equation}\label{Eqn:metricTD}
g^{T\mathring{\mathcal{D}}}=dr^2 \oplus r^2 g^{T_V\mathcal{S}}\oplus g^{T_V\mathcal{S}},
\end{equation} 
where $r>0$ denotes the radial direction. From the orthonormal basis of $T^*\mathcal{S}$ described above we can construct an orthonormal basis for $T^*\mathring{\mathcal{D}}$ as
\begin{align}\label{Def:BasisTD}
\{dr\}\cup\{\widehat{\underline{e}}^i\}_{i=0}^{v}\cup\{\widehat{\underline{f}}^\alpha\}_{\alpha=1}^{h}, 
\end{align} 
setting $\widehat{\underline{e}}^i\coloneqq r\underline{e}^i$ and $\widehat{\underline{f}}^\alpha \coloneqq \underline{f}^\alpha$. Here we regard $\underline{e}^i$ and $\underline{f}^\alpha$ as $1$-forms on $\mathring{\mathcal{D}}$ by pulling them back from $\mathcal{S}$ along the projection $\mathring{\mathcal{D}}\longrightarrow \mathcal{S}$. The generating vector field of the $S^1$-action on $\mathring{\mathcal{D}}$ is still $V_{\mathring{\mathcal{D}}}=\underline{e}_0=r\widehat{\underline{e}}_0$ and therefore, by Proposition \ref{Prop:NormV}(2), we get
 \begin{align*}
\kappa_{\mathring{\mathcal{D}}}=-d\log(\norm{V_{\mathring{\mathcal{D}}}})=-d(\log r)=-\frac{dr}{r}.
\end{align*}

Now we want to compute the $2$-form $\varphi_{0,\mathring{\mathcal{D}}}$. First note that the associated characteristic $1$-form is $\chi_{\mathring{\mathcal{D}}}=\widehat{\underline{e}}^0=r\underline{e}^0$, thus we need to calculate $d\widehat{\underline{e}}^0$ in order to use \eqref{Eqn:dchi}.  Let $\widehat{\underline{\omega}}$ be the connection $1$-form corresponding to the Levi-Civita connection of the metric \eqref{Eqn:metricTD} associated with the basis \eqref{Def:BasisTD}. Using the structure equations 
\begin{align*}
d\widehat{\underline{e}}^i+\widehat{\underline{\omega}}^i_j\wedge \widehat{\underline{e}}^j+\widehat{\underline{\omega}}^i_\alpha\wedge \widehat{\underline{f}}^\alpha+\widehat{\underline{\omega}}^i_r\wedge dr=0,\\
d\widehat{\underline{f}}^\alpha+\widehat{\underline{\omega}}^\alpha_j\wedge \widehat{\underline{e}}^j+\widehat{\underline{\omega}}^\alpha_\beta\wedge \widehat{\underline{f}}^\beta+\widehat{\underline{\omega}}^\alpha_r\wedge dr=0,\\
\widehat{\underline{\omega}}^r_j\wedge\widehat{\underline{e}}^j+\widehat{\underline{\omega}}^r_\alpha\wedge\widehat{\underline{f}}^\alpha=0,
\end{align*}
 we can proceed as in Section \ref{Sect:PfoofSignatureFormula} to obtain the components of $\widehat{\omega}$ (see \eqref{Eqns:Conn1From}), 
\begin{align*}
\widehat{\underline{\omega}}^i_j=&\underline{\omega}^i_j,\\
\widehat{\underline{\omega}}^i_r=&\underline{e}^i,\\
\widehat{\underline{\omega}}^i_\alpha=& r\underline{\omega}^i_\alpha,\\
\widehat{\underline{\omega}}^\alpha_\beta=&r^2\underline{\omega}^\alpha_{\beta i}\underline{e}^i+\underline{\omega}^\alpha_{\beta\gamma}\underline{f}^\gamma,\\
\widehat{\underline{\omega}}^\alpha_r=&0.
\end{align*}
From these equations we find
\begin{align*}
d\chi_{\mathring{\mathcal{D}}}
=&-\underline{\omega}^0_i\wedge(r \underline{e}^i)-(r\underline{\omega}^0_\alpha)\wedge \underline{f}^\alpha-\underline{e}^0\wedge dr\\
=&r\left(-\underline{\omega}^0_i\wedge\underline{e}^i-\underline{\omega}^0_\alpha\wedge \underline{f}^\alpha\right)-\left(-\frac{dr}{r}\right)\wedge (r\underline{e}^0)\\
=&r\varphi_{0,\mathcal{S}}-\kappa_{\mathring{\mathcal{D}}}\wedge \chi_{\mathring{\mathcal{D}}}. 
\end{align*}
As a result, we conclude from \eqref{Eqn:dchi} that $\varphi_{0,\mathring{\mathcal{D}}}=r\varphi_{0,\mathcal{S}}$. We summarize these results in the following proposition. 
\begin{proposition}\label{Prop:KappaPhiDCirc}
Let $F$ be a connected component of the fixed point set of an effective semi-free $S^1$-action on $M$. Then, for the induced action on $\mathring{\mathcal{D}}$, the mean curvature $1$-form is $\kappa_{\mathring{\mathcal{D}}}=-{dr}/{r}$ and the $2$-form $\varphi_{0,\mathring{\mathcal{D}}}$ defined by \eqref{Eqn:dchi} is $\varphi_{0,\mathring{\mathcal{D}}}=r\varphi_{0,\mathcal{S}}$. 
\end{proposition}

Now we proceed to study the $S^1$-quotient. Recall from above that we have the following commutative diagram 
\begin{align*}
\xymatrixrowsep{2cm}\xymatrixcolsep{2cm}\xymatrix{
\mathcal{S} \ar[rr]^-{\pi_{S^1}} \ar[dr]_-{\pi_\mathcal{S}}  && \mathcal{F} \ar[dl]^-{\pi_\mathcal{F}} \\
& F & 
}
\end{align*}
where $\pi_{S^1}:\mathcal{S}\longrightarrow\mathcal{F}$ is the orbit map and $\pi_\mathcal{S}$, $\pi_\mathcal{F}$ are corresponding projections.  The decomposition into the vertical tangent and an horizontal bundle $T\mathcal{S}=T_V\mathcal{S}\oplus T_H\mathcal{S}$ induces a decomposition  $T\mathcal{F}=T_V\mathcal{F}\oplus T_H\mathcal{F}$ via the orbit map $\pi_{S^1}$ (cf. \cite[Section 1]{B09}). Consequently there is an induced splitting of the exterior algebra,
\begin{align}\label{DecompBundleE}
\wedge T^*\mathcal{F}=\wedge T_H^*\mathcal{F}\otimes\wedge T^*_V\mathcal{F} =\bigoplus_{r=p+q}\wedge^p T^*_H\mathcal{F}\otimes \wedge^q T^*_V\mathcal{F}
\eqqcolon
\bigoplus_{p,q}\wedge^{p,q}T^*\mathcal{F}. 
\end{align} 
We denote the space of sections by $\Omega^{p,q}(\mathcal{F})\coloneqq C^\infty(\mathcal{F},\wedge^{p,q}T^*\mathcal{F})$ and the degree operators by $\text{hd}|{\wedge^{p,q} T^*\mathcal{F}}\coloneqq p$ and  $\text{vd}|{\wedge^{p,q} T^*\mathcal{F}}\coloneqq q$.\\
By Remark \ref{Rmk:BasisBasicForms} there exist $1$-forms $\{{e}^i\}_{i=1}^{v}\cup\{{f}^\alpha\}_{\alpha=1}^{h}$, where ${e}^i\in T^*_V\mathcal{F}$ and ${f}^\alpha\in T^*_H\mathcal{F}$ such that $\pi^*_{S^1}{e}^i=\underline{e}^i$ and $\pi^*_{S^1}{f}^\alpha=\underline{f}^\alpha$. The set $\{{e}^i\}_{i=1}^{v}\cup\{{f}^\alpha\}_{\alpha=1}^{h}$ forms a local orthonormal basis for $T^*\mathcal{F}$ and can be regarded as the basis considered in Section \ref{Sect:PfoofSignatureFormula}. We choose an orientation of $\mathcal{F}$ so that  $\{{f_1},\cdots, {f}_h,{e_1},\cdots,{e}_v\}$ is an oriented orthonormal basis. The following result is a direct consequence of Proposition \ref{Prop:KappaPhiDCirc}. 

\begin{proposition}
 Close to a connected component $F\subset M^{S^1}$ of the fixed point set:
 \begin{enumerate}
\item The $1$-form $\bar{\kappa}$ of diagram \ref{Diag:KappaAction} is given by $\bar{\kappa}=-{dr}/{r}$.
\item There exists $\varphi_{0,\mathcal{F}}\in\Omega^2(\mathcal{F})$ such that the $2$-form $\bar{\varphi}_0$ defined by diagram and analogous diagram can be expressed as $\bar{\varphi}_{0}= r\varphi_{0,\mathcal{F}}$. In particular, the operator $\widehat{c}(\bar{\varphi}_0)$
of Proposition \ref{Prop:CliffMultVarphi} is bounded.
\end{enumerate}
 \end{proposition}

Combining this proposition with Theorem \ref{Thm:InducedDiracOp} and the Kato-Rellich Theorem (\cite[Theorem V.4.3]{KATO}) we can prove the claim of Remark \ref{Rmk:OpD}.

\begin{coro}
The operator $\mathscr{D}:\Omega_c(M_0/S^1)\longrightarrow\Omega_c(M_0/S^1)$ defined by 
\begin{align*}
\mathscr{D}\coloneqq D_{M_0/S^1}+\frac{1}{2}c(\bar{\kappa})\varepsilon,
\end{align*}
is essentially self-adjoint. 
\end{coro}

\subsection{Local description of $\mathscr{D}$ }\label{Section:DecConeOp}

Now we analyze how the complete operator $\mathscr{D}$ can be written near the fixed point set. For the Hodge-de Rham operator this was done by Br\"uning in \cite[Section 2]{B09}.  
For $t>0$,  define the model space $U_t\coloneqq \mathcal{F}\times (0,t)$ and let $\pi: U_t\longrightarrow \mathcal{F}$ be the  projection onto the first factor. Equip $U_t$ with the metric
\begin{align}\label{Eqn:MetricUt}
g^{TU_t}\coloneqq  dr^2\oplus  g^{T_H\mathcal{F}}\oplus r^2 g^{T_V\mathcal{F}},
\end{align}
for $0<r<t$. We choose the orientation on $U_t$ defined by the oriented orthonormal basis 
$\{-\partial_r, \widehat{f_1},\cdots, \widehat{f}_h,\widehat{e_1},\cdots,\widehat{e}_v\}$. Consider the unitary transformation introduced in \cite[Equation 2.12]{B09},
\begin{equation*}\label{Eqn:TransfPsi1}
\Psi_1:
\xymatrixcolsep{2cm}\xymatrixrowsep{0.01cm}\xymatrix{
L^2((0,t),\Omega(\mathcal{F})\otimes\mathbb{C}^2) \ar[r] & L^2(U_t)\\
(\sigma_1,\sigma_2) \ar@{|->}[r] & \pi^* r^{\nu}\sigma_1(r)+ dr\wedge\pi^* r^\nu \sigma_2(r),
}
\end{equation*}
where $\nu$ is the operator defined by
$$\nu:=\text{vd}-\frac{v}{2}=\text{vd}-N.$$
For example $\nu e_i=(1-N)e_i$ and $\nu f_\alpha=-Nf_\alpha$. Let us denote the horizontal and vertical chirality operators by
\begin{align*}
\bar{\star}_H\coloneqq &i^{[(h+1)/2]}c({f}^1)\cdots c({f}^h),\\
\bar{\star}_V\coloneqq &i^{[(v+1)/2]}c({e}^1)\cdots c({e}^v),
\end{align*}
A straightforward computation shows that we can write the transformed chirality operator as (\cite[Lemma 2.4]{B09}, \cite[Remark 6.6]{JO17}))
\begin{align}\label{TransTauPsi1}
\widehat{\star}\coloneqq 
\Psi_1^{-1}\bar{\star}\Psi_1
=\left(\begin{array}{cc}
0 & -\bar{\star}_H\bar{\star}_V \\
-\bar{\star}_H\bar{\star}_V  & 0
\end{array}
\right)
=\left(\begin{array}{cc}
0 & I\\
I & 0
\end{array}
\right)\otimes(-\alpha),
\end{align}
where $\alpha \coloneqq \varepsilon_H^v \bar{\star}_H\otimes\bar{\star}_V=\bar{\star}_H\bar{\star}_V $ and $\varepsilon_H/\varepsilon_V$ denote the horizontal/vertical Gau\ss-Bonnet involution. It is easy to verify from \eqref{TransTauPsi1} that the unitary transformation
\begin{align*}
\mathcal{U}\coloneqq
\frac{1}{\sqrt{2}}
\left(
\begin{array}{cc}
I & \alpha\\
-\alpha & I 
\end{array}
\right)
\end{align*}
diagonalizes $\widehat{\star}$ (\cite[Equation 2.27]{B09}), i.e.
\begin{align*}
\bigstar\coloneqq
\mathcal{U}^{-1}\hat{\star}\mathcal{U}
\left(
\begin{array}{cc}
I & 0 \\
0 & -I
\end{array}
\right).
\end{align*}

\begin{theorem}[{\cite[Theorem 2.5]{B09}}]\label{Thm:TransfDiracOps}
Under the unitary transformation  $\Psi\coloneqq \Psi_1 \mathcal{U}$ the Hodge-de Rham operator $D_{U_t}$  on $U_t$ is transformed as
\begin{equation*}
\Psi^{-1}D_{U_t}\Psi=\gamma\left(\frac{\partial}{\partial r}+
\bigstar\otimes A(r)
\right). 
\end{equation*}
In particular the  transformed signature operator is
\begin{equation*}
\Psi^{-1}D^{\textnormal{+}}_{U_t}\Psi=\frac{\partial}{\partial r}+A(r). 
\end{equation*}
The operator $A(r)$ is defined by $A(r)\coloneqq A_H(r)+r^{-1}A_V$. Here
\begin{align*}
A_H(r)\coloneqq &\left((d^{(1)}_H +r d^{(2)}_H) + (d^{(1)}_H + rd^{(2)}_H)^\dagger \right)\alpha,\\
A_{V}\coloneqq & (d_V + d^\dagger_V)\alpha+\nu.
\end{align*}
The operators $d_V, d^{(1)}_H$ and $d^{(2)}_H$ are obtained from the decomposition of the exterior derivative $d_{U_t}$ with respect to \eqref{DecompBundleE} (cf. \cite[Lemma 2.1]{B09}). The first order vertical differential operator $A_V$ is called the cone coefficient. 
\end{theorem}

The following result, which is of fundamental importance for later purposes, is a consequence of Theorem \ref{Thm:TransfDiracOps} and the discussion of \cite[Section 1]{B09}.
\begin{lemma}[{\cite[Theorem 2.5]{B09}}]\label{Lemma:VertOp1}
The operator $A_{HV}\coloneqq A_H(0)A_V+A_VA_H(0)$ is a first order vertical operator, i.e. it only differentiates with respect to the vertical coordinates. If $A_V$ is invertible then, for $r$ small enough, there exists a constant $C>0$ such that $A(r)^2\geq Cr^{-2}A_V^2$, in particular $A(r)$ is also invertible.
\end{lemma}

Using the relation $\varepsilon\alpha=-\alpha\varepsilon$ one verifies the transformation law
\begin{align*}
\Psi^{-1}\left(
-\frac{1}{2r}c(dr)\varepsilon
\right) \Psi=
\gamma\left(
\frac{1}{2r}
\left(
\begin{array}{cc}
-\varepsilon & 0\\
0 & \varepsilon
\end{array}
\right)
\right).
\end{align*}
Combining this formula and Theorem \ref{Thm:TransfDiracOps} we obtain the following local description of $\mathscr{D}$ close to the fixed point set. 
\begin{theorem}\label{Thm:LocalDesD1}
Under the unitary transformation $\Psi$ the operator $\mathscr{D}_{U_t}$ transforms as 
\begin{equation*}
\Psi^{-1}\mathscr{D}_{U_t}\Psi=\gamma\left(\frac{\partial}{\partial r}+
\bigstar\otimes \mathscr{A}(r)
\right),
\end{equation*}
where $\mathscr{A}(r)\coloneqq A_H(r) + r^{-1}\mathscr{A}_V$ and the associated cone coefficient is
\begin{align*}
\mathscr{A}_V\coloneqq A_V-\frac{1}{2}\varepsilon = 
A_V - \frac{1}{2r}\varepsilon_H\otimes\varepsilon_V.
\end{align*}
\end{theorem}

\subsection{Spectral decomposition of the cone coefficient}\label{Section:SpectralDecomp}

The cone coefficient $A_V$ is of fundamental importance for the study of the Hodge-de Rham operator $D_{M_0/S^1}$. The resolvent analysis by Br\"uning and Seeley (\cite{B09}, \cite{BS91}) shows that if $|A_V|\geq 1/2$ then $D_{M_0/S^1}$ is discrete and essentially self-adjoint (\cite[Theorem 0.1]{B09}). Therefore, the understanding of the spectrum of $A_V$ is crucial for the theory. A detailed analysis on the subject is presented in \cite[Section 3]{B09}. The main idea behind this analysis is to study $\spec(A_V)$ on the space of vertical harmonic forms $\mathcal{H}$ and its complementary space. On the first space the eigenvalues of $A_V$ are given by $2j-N$ for $j=0,1, \cdots, N$ (here one uses the fact that the fiber $\mathbb{C}P^{N}$ has non-trivial cohomology groups in even dimensions). On the complement the eigenvalues have the form $\pm \sqrt{\lambda + (j-N-1/2)^2}$ for  $j=0,1, \cdots, 2N$ and $\lambda > 0$ is a eigenvalue of the vertical Laplacian. By rescaling the vertical component of the metric \eqref{Eqn:MetricUt} one can make the eigenvalues $\lambda$ as large as necessary so that the spectrum condition of $A_V$ is satisfied (\cite[Remark 2.6]{JO17}). Moreover, this rescaling does not change the index of $D_{M_0/S^1}$ (\cite[pg. 29-30]{B09}). This shows it is enough to study $\spec(A_V)$ on the space of vertical harmonic forms if we are interested in the index computation. We have two possible cases: if $N$ is odd (Witt case) then $|2j-N|\geq 1$ and if $N$ is even (non-Witt case)
we get a zero eigenvalue on $\mathcal{H}^{N/2}$. Hence, in the Witt case, by recaling the vertical metric if necessary, the Hodge-de Rham operator $D_{M_0/S^1}$ is essentially self-adjoint. In the non-Witt case we are forced to impose {\em boundary conditions}. \\

From Theorem \ref{Thm:InducedDiracOp} we know that $\mathscr{D}$ is essentially self-adjoint independently on the parity of $N$. We can also see this from the spectrum of the cone coefficient $\mathscr{A}_V$. As before, up to rescaling, it is enough to study $\spec(\mathscr{A})$ restricted to $\mathcal{H}$. It follows directly from Theorem \ref{Thm:LocalDesD1} that these eigenvalues are of the form $2j-N\pm1/2$  for $j=0,1, \cdots, N$ (\cite[Theorem 6.22]{JO17}). In particular, we see $|\mathscr{A}_V|\geq 1/2$. In addition, due the nature of the potential of $\mathscr{D}$, the resolvent analysis goes along the same lines as for $D_{M_0/S^1}$ and one can show that $\mathscr{D}$ is also discrete (\cite[Section 7.2]{JO17}).

\section{An index theorem}\label{Sect:Index}

In this section we compute the index of the operator $\mathscr{D}^+$, at least in the Witt case, following the techniques used in \cite[Section 5]{B09}. The analytic tools required to treat the index computation are based on the work \cite{BBC08} of Ballmann, Br\"uning and Carron on {\em Dirac-Schr\"odinger Systems}. We begin by recalling some results on perturbation of regular projections (\cite[Sections 1.2, 1.6]{BBC08} and \cite[Section 5]{B09}). 

\subsection{Perturbation of regular projections}\label{Sec:Perturbations}

Let $(H, \inner{\cdot}{\cdot})$ be a separable complex Hilbert space and let $A:\dom(A)\subset H\longrightarrow H$ be a discrete self-adjoint operator. For each Borel subset $J\subseteq\mathbb{R}$ we denote by $Q_J\coloneqq Q_J(A)$ the associated spectral projection of A in $H$. For $\Lambda\in\mathbb{R}$ we will use the notation $Q_{<\Lambda}\coloneqq Q_{(-\infty,\Lambda)}$, $Q_{\geq\Lambda }\coloneqq Q_{[\Lambda,\infty)}$, etc. For $\Lambda=0$ we write $Q_{>}\coloneqq Q_{>0}$, etc. In particular $Q_0$ denotes the projection onto $\ker(A)$.

\begin{definition}[Sobolev chain]
For $s\geq 0$ we define the space $H^s:=H^s(A)$ to be the completion of $\dom(A)\subset H$ with respect to the inner product 
\begin{align*}
\inner{\sigma_1}{\sigma_2}_s\coloneqq \inner{(I+A^2)^{s/2}\sigma_1}{(I+A^2)^{s/2}\sigma_2}. 
\end{align*}
Note for example that $H^0=H$ and $H^1=\dom(A)$. For $s<0$ we define $H^s$ to be the strong dual space of $H^{-s}$.
\end{definition}
For a Borel subbset $J\subseteq\mathbb{R}$ we denote by $H^s_J\coloneqq Q_J(H^s)$ the image of the Sobolev space $H^s$  under the projection $Q_J$. In particular we will use the notation
$H^s_<\coloneqq Q_{<0}(H^s)$, etc.

\begin{definition}
A bounded operator $S\in\mathcal{L}(H)$ is called {\em $1/2$-smooth} if it restricts to an operator $\hat{S}:H^{1/2}\longrightarrow H^{1/2}$ and extends to an operator $\tilde{S}:H^{-1/2}\longrightarrow H^{-1/2}$. The operator $S$ will be called {\em $(1/2)$-smoothing}, or simply {\em smoothing}, if $\ran(\tilde{S})\subset H^{1/2}$.
\end{definition}

In addition to the operator $A$ let us assume that we are given  $\gamma\in\mathcal{L}(H)$ such that 
\begin{enumerate}
\item $\gamma^*=\gamma^{-1}=-\gamma$.
\item $\gamma A+A\gamma=0$.
\end{enumerate}

\begin{definition}
A $1/2$-smooth orthogonal projection $P$ in $H$ is called {\em regular} (with respect to A) if for some, or equivalently, for any $\Lambda\in\mathbb{R}$ we have 
\begin{align*}
\sigma\in H^{-1/2}, \tilde{P}\sigma=0, Q_{\leq \Lambda}(A)\sigma\in H^{1/2}\:\Rightarrow\: \sigma\in H^{1/2}.
\end{align*}
A regular projection $P$ is called {\em elliptic} if $P_\gamma\coloneqq \gamma^*(1-P)\gamma$ is also a regular projection. 
\end{definition}

\begin{example}[Spectral projections]\label{Example:SpectralProjections}
Let $\Lambda\in\mathbb{R}$, then the spectral projection $Q_{>\Lambda}$ is an elliptic projection. 
\end{example}

Let $B$ be a symmetric operator defined on $\dom(A)$ such that 
$\norm{B\sigma}\leq a\norm{\sigma}+b\norm{A\sigma}$ for all $x\in\dom(A)$ and $a,b\in\mathbb{R}_+$ with $b<1$. We call the operator $A+B$, defined on $\dom(A)$, a {\em Kato perturbation of $A$}. By Kato-Rellich Theorem the operator $A+B$ is again self-adjoint and discrete (\cite[Theorem V.4.3]{KATO}).  Thus, we can consider the corresponding Sobolev chain $H^s(A+B)$. Moreover, we can identify as Hilbert spaces $H^s(A)\cong H^s(A+B)$ for all $s\in \mathbb{R}$ (\cite[Remark 8.15]{JO17}). The following result describes how the ellipticity condition of a spectral projection behaves under a Kato perturbation. 
\begin{theorem}[{\cite[Theorem 5.9]{B09}}]\label{BThm5.9}
Assume that $A+B$ is a Kato perturbation of $A$ with $b<2/3$. Then $Q_{>}(A+B)$ is an elliptic projection with respect to $A$ and the subspaces $Q_{\leq }(A)(H)\coloneqq \ran(Q_\leq (A))$ and $Q_{> }(A+B)(H)\coloneqq \ran(Q_>(A+B))$ from a Fredholm pair. If $B$ is bounded and $|A|\geq \mu$ where $\mu>\sqrt{2}\norm{B}$, then 
$\textnormal{\ind}(Q_{\leq }(A)(H),Q_{>}(A+B)(H))=0.$
\end{theorem}

\subsection{The index formula for the signature operator in the Witt case}\label{Section:IndexSigOpWitt}

The objective of this section is to study the techniques used in \cite[Section 5]{B09} to compute the index of the signature operator ${D}^+$ on $M_0/S^1$ in the Witt case. The strategy, in view of \cite[Theorem 4.17]{BBC08}, is to consider compatible {\em super-symmetric Dirac systems} over $Z_t$ and $U_t$ (see Figure \ref{Fig:DecompQuotSpace}) in the sense of \cite[Section 3]{BBC08} and calculate the index contribution on each piece separately as $t\longrightarrow 0$. In the subsequent section we will then adapt these same techniques to compute the index of the operator $\mathscr{D}^+$ also in the Witt case. \\

In view of Theorem \ref{Thm:TransfDiracOps} we define, for $0<t<t_0/2$ fixed, the operator
\begin{align*}
{D}_1\coloneqq \gamma\left(\frac{\partial}{\partial r}+\bigstar\otimes {A}(t+r)\right),
\end{align*}
where $0<r<t$. From the relation $\gamma\bigstar +\bigstar\gamma =0$ we see that $({D}_1,\bigstar)$  defines a super-symmetric Dirac system on $Z_t$. The corresponding graded operator is 
\begin{align*}
{D}^+_1=\frac{\partial}{\partial r}+{A}(t+r). 
\end{align*}
\begin{remark}
The Hilbert space on which ${A}(t)$ is defined is $H\coloneqq L^2(\wedge T^*\mathcal{F}_t)$ where the metric on $\mathcal{F}_t$ is $g^{T\mathcal{F}_t}=g^{T_H\mathcal{F}}\oplus t^2 g^{T_V\mathcal{F}}$. For $s\geq 0$ the associated Sobolev chain is $H^{s}\coloneqq \dom(|{A}(t)|^s)$. 
\end{remark}
Similarly, we define on $U_t$,
\begin{align*}
{D}_2\coloneqq -\gamma\left(\frac{\partial}{\partial r}-\bigstar\otimes {A}(t-r)\right),
\end{align*}
and we also see that $(D_2,\bigstar)$ is a super-symmetric Dirac system with associated graded operator 
\begin{align*}
{D}^+_2=-\left(\frac{\partial}{\partial_r}-{A}(t-r)\right).
\end{align*}

It is straightforward to verify that the super-symmetric Dirac systems ${D}_1=({D}_1,\bigstar)$ and ${D}_2=({D}_2,\bigstar)$ are compatible in the sense of \cite[Section 3]{BBC08} . We should think of ${D}_1$ and ${D}_2$ as the restrictions of the operator ${D}$ to $Z_t$ and $U_t$ respectively. 
For the operator ${D}_1$ we impose the APS-type boundary condition $B_1\coloneqq \bigstar Q_{<}({A}(t))(\dom(|{A}(t)|^{1/2}))$, which we know is elliptic. 
Obviously this boundary is invariant under $\bigstar$, so we can split it as $B_1=B^+_1\oplus B^-_1$ where
$B^+_1\coloneqq  Q_{<}({A}(t))(\dom(|{A}(t)|^{1/2}))$. For the operator ${D}^+_2$  we choose the complementary boundary condition  $B^+_2\coloneqq  Q_{\geq}({A}(t))(\dom(|{A}(t)|^{1/2}))$. If we denote both operators with their compatible elliptic boundary conditions by
\begin{align*}
{D}^{+}_{Z_t,Q_{<}({A}(t))(H)}\coloneqq &{D}^+_{1,{B^+_1}},\\
{D}^{+}_{U_t,Q_{\geq}({A}(t))(H)}\coloneqq &{D}^+_{1,{B^+_2}},
\end{align*}
then from \cite[Theorem 4.17]{BBC08} we obtain the following decomposition result for the index.

\begin{theorem}[{\cite[Theorem 5.1]{B09}}]\label{Thm:5.1}
The operators ${D}^+_{U_t,Q_{\geq}({A}(t))(H)}$ and ${D}^+_{Z_t,Q_{<}({A}(t))(H)}$ are Fredholm, and we have the index identity
\begin{align*}
\textnormal{ind}({D}^+)=\textnormal{ind}\left({D}^{+}_{Z_t,Q_{<}({A}(t))(H)}\right)+\textnormal{ind}\left({D}^{+}_{U_t,Q_{\geq }({A}(t))(H)}\right).
\end{align*}
\end{theorem}

We first study the index contribution on $U_t$. Set ${D}^+_t\coloneqq  {D}_{U_t,Q_{\geq}({A}(t))(H)}$ to lighten the notation. From the discussion in Section \ref{Section:SpectralDecomp} we know that
\begin{align*}
\mathcal{C}({D}^+_t)\coloneqq \{\sigma\in C^{1}_c((0,t],H^1))\:|\:Q_<({A}(t))\sigma(t)=0\}
\end{align*}
is a core for ${D}^+_t$. One of the main ingredients of the Br\"uning's index computation in \cite{B09} is the following remarkable vanishing result. 
\begin{theorem}[{\cite[Theorem 5.2]{B09}}]\label{Thm:5.2}
For $t\in(0,t_0]$ sufficiently small we have 
\begin{align*}
\textnormal{ind}\left({D}^{+}_{U_t,Q_{\geq}({A}(t))(H)}\right)=0. 
\end{align*}
\end{theorem}

The key point of the proof is the fact that, by the rescaling argument and since in our particular case the dimension of the fibers is even, we can always assume that 
\begin{align}\label{Eqn:1/2notinSpecAV}
\pm\frac{1}{2}\notin\spec(A_V),
\end{align}
This observation allows us to obtain an estimate which shows that $\ker(D^+_t)=\{0\}$.  Then, arguing analogously for the adjoint operator one verifies the vanishing of the index. The concrete form of the claimed estimate is discussed in the following lemma. We revise its proof because it will inspire techniques to derive a similar vanishing result for the operator $\mathscr{D}^+$. 

\begin{lemma}\label{Lemma:5.12}
For $t$ small enough and $\sigma\in\mathcal{C}({D}^+_t)$ we have the estimate
\begin{align*}
\norm{{D}^+_{t}\sigma}_{L^{2}((0,t],H)}\geq \frac{1}{2t}\bnorm{\left({A}_V+\frac{1}{2}\right)\sigma}_{L^{2}((0,t],H)}.
\end{align*}
\end{lemma}

\begin{proof}
Let $\sigma\in\mathcal{C}({D}^+_{t})$. 
We use the explicit form of $A(r)$ to compute
\begin{align*}
\norm{{D}^+_{t}\sigma(r)}_{H}^2=&\norm{\sigma'(r)}^2_{H}+\norm{A_H\sigma(r)}^2_{H}+r^{-2}\norm{{A}_V\sigma(r)}^2_{H}+r^{-1}\inner{{A}_{HV}\sigma(r)}{\sigma(r)}_{H}\\
&-\frac{d}{dr}\inner{\sigma(r)}{{A}(r)\sigma(r)}_H+r^{-2}\inner{\sigma(r)}{{A}_V\sigma(r)}-\inner{\sigma(r)}{A'_H(0)\sigma(r)}_H. 
\end{align*}

From Lemma \ref{Lemma:VertOp1} it follows that ${L}\coloneqq {A}_{HV}(r)({A}_V+1/2)^{-1}$ is a zero order operator and  therefore there exists a constant $C_1>0 $ such that,
\begin{align*}
\bnorm{{A}_{HV}(r)\left({A}_V+\frac{1}{2}\right)^{-1}}_H\leq C_1, \:\text{for $r\in(0,t]$}.
\end{align*}
It is straightforward to see that we can write
\begin{align}\label{Eqn:AuxLemma5.12}
\norm{{D}^+_{t}\sigma(r) }_{H}^2=&\left(\norm{\sigma'(r)}^2_H-\frac{1}{4r^2}\norm{\sigma(r)}^2_H\right)+\frac{d}{dr}\inner{\sigma(r)}{{A}(r)\sigma(r)}_H\\
&+r^{-2}\bnorm{\left({A}_V+\frac{1}{2}\right)\sigma(r)}^2_H+r^{-1}\left\langle\left({A}_V+\frac{1}{2}\right)\sigma(r),{L}^*\sigma(r)\right\rangle_H\notag\\
&-\inner{\sigma(r)}{A'_H(0)\sigma(r)}_H+\norm{A_H(r)\sigma(r)}^2.\notag
\end{align}
Now we integrate \eqref{Eqn:AuxLemma5.12} between $0$ and $t$ in order to compute the $L^2$-norm. After integration, the first term in the right hand side of this equation is positive by Hardy's inequality (\cite{H20}). The second term in the right hand side of \eqref{Eqn:AuxLemma5.12} is also positive after integration as a result of the boundary condition at $t$. Finally, choose $t$ small enough such that 
\begin{align*}
\norm{{D}^+_{t}\sigma}^2_{L^2((0,t],H)} \geq \frac{1}{4}\int_0^t r^{-2}\bnorm{\left({A}_V+\frac{1}{2}\right)\sigma(r)}^2_H dr\geq \frac{1}{4t^2}\bnorm{\left(A_V+\frac{1}{2}\right)\sigma}^2_{L^2((0,t],H)}.
\end{align*}
\end{proof}

\begin{remark}\label{Lemma:5.12R}
Observe that the adjoint operator 
\begin{align*}
({D}^+_{t})^*=-\frac{\partial}{\partial r}+{A}(r),
\end{align*}
has  a core $\mathcal{C}(({D}^+_{t})^*)\coloneqq \{\sigma\in C^1_c((0,t],H^1)\:|\:Q_{\geq}({A}(t))\sigma(t)=0\}$.
Hence,  we can compute similarly for $\sigma\in \mathcal{C}(({D}^+_{t})^*)$,
 All together, we get the analogous estimate
\begin{align*}
\norm{({D}^+_{t})^*\sigma}_{L^{2}((0,t],H)}\geq \frac{1}{2t}\bnorm{\left({A}_V-\frac{1}{2}\right)\sigma}_{L^{2}((0,t],H)}.
\end{align*}
\end{remark}

From Theorem \ref{Thm:5.2} we conclude that for $t>0$ small enough
\begin{align}\label{Eqn:Indexwoc}
\textnormal{ind}({D}^+)=\textnormal{ind}\left({D}^{+}_{Z_t,Q_{<}({A}(t))(H)}\right),
\end{align}
which is just the index of the signature operator on the manifold with boundary $Z_t$ with an APS-type boundary condition. 
In order to compute this index we would like to use \cite[Theorem 4.14]{APSI}. Moreover, note that the left hand side of \eqref{Eqn:Indexwoc} does not depend on $t$, thus we can study the behavior of the right hand side in the limit $t\longrightarrow 0$. This is of course motivated by the proof of Theorem \ref{Thm:S1SignatureThm}. Nevertheless, we need to be aware of the following observations:

\begin{enumerate}
\item The metric close to the boundary is not a product. Still, we can modify it so that it becomes a product near  $r=t$ without changing the index (\cite[pg. 32]{B09}).  
\item   Note that $Q_\geq({A}(t))(H)$ is not the appropriate boundary used to derive \cite[Theorem 4.14]{APSI} since ${A}(t)$ does not correspond to the tangential signature operator on $\partial Z_t$. To correct this we just need to subtract a bounded term, 
\begin{align*}
A_0(t)\coloneqq A(t)  - \frac{\nu}{t}. 
\end{align*}
\end{enumerate}

Combining  observation (2) with the results of \cite[Section 4]{APSI} and Section \ref{Sect:PfoofSignatureFormula} we obtain the formula
\begin{align}\label{IndS1SignKern}
\sigma_{S^1}(M)=\textnormal{ind}\left({D}^+_{Z_t,Q_{<}({A}_0(t))(H)}\right)+\frac{1}{2}\dim\ker(A_0(t)). 
\end{align}

The next result shows how the index changes when varying boundary conditions. 
\begin{theorem}[{\cite[Theorem 5.3]{B09}}]\label{Thm:5.3}
For $t$ sufficiently small $(Q_{<}({A}(t))(H),Q_{\geq }({A}_0(t))(H))$ is a Fredholm pair in $H$ and 
\begin{align*}
\textnormal{ind}\left( {D}^+_{Z_t,Q_{<}({A}(t))(H)}\right)=&\textnormal{ind}\left( {D}^+_{
Z_t,Q_{<}(A_0(t))(H)}\right)
+\ind(Q_{<}({A}(r))(H),Q_{\geq }({A}_0(r))(H)).
\end{align*}
\end{theorem}

This theorem, combined with \eqref{Eqn:Indexwoc} and \eqref{IndS1SignKern}, implies the following result. 
\begin{proposition}\label{Prop:IndexFormulaCasi}
For $t>0$ sufficiently small we have
\begin{align*}
\ind ({D}^+)=\sigma_{S^1}(M)-\frac{1}{2}\dim\ker(A_0(t))+\ind(Q_{<}({A}(t))(H),Q_{\geq }({A}_0(t))(H)).
\end{align*}
\end{proposition}

The computation of the Kato index in the proposition above is highly non-trivial.  The main ingredient is the {\em generalized Thom space} $T_\pi$ associated to of the fibration  $\pi_\mathcal{F}:\mathcal{F}\longrightarrow F$ (cf. \cite{CD09}). As we are assuming the Witt condition on $M/S^1$ then this is also true for the compact stratified space $T_\pi$. In particular it has a well-defined signature operator whose index computes the $L^2$-signature $ \sigma_{(2)}(T_\pi)$. Using the vanishing result \cite[Theorem 5.2]{B09} applied to $T_\pi$ Br\"uning proved the following remarkable identity (\cite[Theorem 5.4]{B09})
\begin{align*}
\ind(Q_{<}({A}(t))(H),Q_{\geq }(A_0(t))(H))=\sigma_{(2)}(T_\pi)+\frac{1}{2}\dim\ker(A_0(t)).
\end{align*}

On the other hand, Cheeger and Dai showed, still restricted to the Witt case, that this $L^2$-signature coincides with Dai's $\tau$ invariant of the fibration $\pi_\mathcal{F}:\mathcal{F}\longrightarrow F$, which is this case vanishes (cf. Section \ref{Sect:PfoofSignatureFormula}). 
We therefore conclude from Proposition \ref{Prop:IndexFormulaCasi} and Theorem \ref{Thm:S1SignatureThm} that the index of the signature operator on $M_0/S^1$, in the Witt case, computes Lott's equivariant $S^1$-signature. 
\begin{theorem}\label{Thm:IndD+}
Let $M$ be a closed, oriented Riemannian $4k+1$ dimensional manifold on which $S^1$ acts effectively and semi-freely by orientation preserving isometries. If the codimension of the fixed point set $M^{S^1}$ in $M$ is divisible by four, then $M/S^1$ is a Witt space and the index of the signature operator is 
\begin{align*}
\ind(D^+)=\sigma_{S^1}(M)=\int_{M_0/S^1}L\left(T(M_0/S^1),g^{T(M_0/S^1)}\right).
\end{align*}
\end{theorem}
Observe that we have used the fact that the $\eta(M^{S^1})$ vanishes in the Witt case.

\subsection{The index formula for the Dirac-Schr\"odinger signature operator}

In this last section we describe how to compute the index of the operator $\mathscr{D}^+$ using the techniques illustrated in Section \ref{Section:IndexSigOpWitt}. We obtain the complete index formula for this operator in the Witt case. To begin with, we make no assumption on the parity of $N$, i.e. we do not distinguish between the Witt and the non-Witt case. As before, we use the geometric decomposition of $M_0/S^1$ in order to construct compatible super-symmetric Dirac systems  
\begin{align*}
\mathscr{D}_1\coloneqq &\gamma\left(\frac{\partial}{\partial r}+\bigstar\otimes \mathscr{A}(t+r)\right),\\
\mathscr{D}_2\coloneqq &-\gamma\left(\frac{\partial}{\partial r}-\bigstar\otimes \mathscr{A}(t-r)\right),
\end{align*}
where $\mathscr{A}(r)$ is the operator of Theorem \ref{Thm:LocalDesD1}. The corresponding graded operators, with respect to the super-symmetry $\bigstar$, are 
\begin{align*}
\mathscr{D}^+_1&=\frac{\partial}{\partial r}+\mathscr{A}(t+r),\\
\mathscr{D}^+_2&=-\left(\frac{\partial}{\partial r}-\mathscr{A}(t-r)\right).
\end{align*}
\begin{remark}\label{Rmk:EquivSobChainsPot}
In this case we still set $H\coloneqq L^2(\wedge T^*\mathcal{F}_t)$ and consider the associated Sobolev chain of $\mathscr{A}(t)$  denoted by $H^{s}\coloneqq \dom(|\mathscr{A}(t)|^s)$. Since for $t>0$ fixed, $\mathscr{A}(t)-A(t)$ 
is a bounded operator, it follows by \cite[Remark 8.15]{JO17} that the Sobolev chains of $\mathscr{A}(t)$ and $A(t)$ are isomorphic.
\end{remark}

As we did for the signature operator, we impose the complementary APS-type boundary conditions
$\mathscr{B}^+_1\coloneqq  Q_{<}(\mathscr{A}(t))(\dom(|\mathscr{A}(t)|^{1/2}))$ and
$\mathscr{B}^+_2\coloneqq  Q_{\geq}(\mathscr{A}(t))(\dom(|\mathscr{A}(t)|^{1/2}))$
on $Z_t$ and $U_t$ respectively. If we denote both operators with their compatible elliptic boundary conditions by
\begin{align*}
\mathscr{D}^{+}_{Z_t,Q_{<}(\mathscr{A}(t))(H)}\coloneqq &\mathscr{D}_{1,{\mathscr{B}^+_1}},\\
\mathscr{D}^{+}_{U_t,Q_{\geq}(\mathscr{A}(t))(H)}\coloneqq &\mathscr{D}_{1,{\mathscr{B}^+_2}},
\end{align*}
then, as in last section, we can apply \cite[Theorem 4.17]{BBC08} to obtain decomposition formula 
\begin{align}\label{Eqn:Thm5.1OpPot}
\textnormal{ind}(\mathscr{D}^+)=\textnormal{ind}\left(\mathscr{D}^{+}_{Z_t,Q_{<}(\mathscr{A}(t))(H)}\right)+\textnormal{ind}\left(\mathscr{D}^{+}_{U_t,Q_{\geq }(\mathscr{A}(t))(H)}\right).
\end{align}

\subsubsection{The index formula for $\mathscr{D}^+$ in the Witt case}
Now we restrict ourselves to the Witt case, i.e. $N$ is odd. In order to obtain a vanishing result for the index contribution of $U_t$ we need to modify the boundary conditions at $r=t$. As noted in Remark \ref{Rmk:EquivSobChainsPot}, for fixed $t>0$ the operator $A(t)$ is a Kato perturbation of $\mathscr{A}(t)$. Hence, by Theorem \ref{BThm5.9} and Theorem \cite[Theorem 4.17]{BBC08} we obtain the decomposition formula of the index 
\begin{align}\label{Eqn:Thm5.1OpPotI}
\textnormal{ind}(\mathscr{D}^+)=\textnormal{ind}\left(\mathscr{D}^{+}_{Z_t,Q_{<}({A}(t))(H)}\right)+\textnormal{ind}\left(\mathscr{D}^{+}_{U_t,Q_{\geq }({A}(t))(H)}\right).
\end{align}

The following lemma relates these two boundary conditions at $r=t$. 
\begin{lemma}\label{Lemma:SplitKatoInd}
The following index identity holds true, 
\begin{align*}
 \ind(Q_{<}  (\mathscr{A}(t))(H)  ,  Q_{\geq }(A_0(t))(H))
=&\: \ind(Q_{<}({A}(t))(H),Q_{\geq }(A_0(t))(H))\\
&+\ind(Q_{<}(\mathscr{A}(t))(H),Q_{\geq }(A(t))(H)).
\end{align*}
\end{lemma}

\begin{proof}
Using similar arguments as in the proof of Theorem \ref{Thm:5.3} one can see that 
\begin{align*}
&(Q_{<}({A}(t))(H),Q_{\geq }(A_0(t))(H)),\\
&(Q_{<}(\mathscr{A}(t))(H),Q_{\geq }(A(t))(H)),
\end{align*} 
are both Fredholm pairs. Now we apply \cite[Lemma A.1]{BB01} with $A_1=\mathscr{A}(t)$, $A_2=A_0(t)$,  $\alpha_1=\alpha_2=0$ to verify that
 \begin{align*}
 Q_>(\mathscr{A}(t))-Q_>(A_0(t))=\frac{1}{2\pi i}\int_{\text{Re}\:z=0}(\mathscr{A}(t)-z)^{-1}\left[\frac{1}{t}\left(\nu-\frac{\varepsilon}{2}\right)\right](A_0(t)-z)^{-1}dz,
 \end{align*}
is compact for fixed $t>0$. Indeed, both $\mathscr{A}(t)$ and $A_0(t)$ are discrete, so their resolvent is compact, and the perturbation term $(\nu-\varepsilon/2)$ is a bounded operator. Similarly, we see that all the differences
\begin{align*}
 &Q_<(\mathscr{A}(t))-Q_<(A_0(t)),\\
 &Q_>(\mathscr{A}(t))-Q_>(A(t)),\\
  &Q_>({A}(t))-Q_>(A_0(t)),
\end{align*}
are compact. Finally we can use Lemma \cite[Lemma A.1]{BB01} with projections $P=Q_{<}(\mathscr{A}(r))$, $Q=Q_{<}({A}(r))$ and $B=Q_{\geq }(A_0(r))(H)$ to obtain the desired formula. 
\end{proof}

\begin{lemma}
For $t>0$ sufficiently small we have
$
\ind(Q_\leq(A(t))(H),Q_>(\mathscr{A}(t)(H))=0.
$
\end{lemma}

\begin{proof}
The idea is to apply Theorem \ref{BThm5.9} with $A=A(t)$ and $B=\varepsilon/2t$ so that the sum $A+B=\mathscr{A}(t)$. By Lemma \ref{Lemma:VertOp1} we see that for $t$ small enough, 
\begin{align*}
|A(t)|\geq\frac{\sqrt{C}}{t}|A_V|\geq \frac{\sqrt{C}}{t}.
\end{align*}
Hence, the required condition for $\mu\coloneqq \sqrt{C}/t$ is 
\begin{align*}
\frac{\sqrt{C}}{t}>\sqrt{2}\bnorm{\frac{\varepsilon}{2t}}=\frac{\sqrt{2}}{2t},
\end{align*}
that is, $C>1/2$, which we can be achieved.
\end{proof}

From this lemma we conclude, via \cite[Theorem 4.14]{BBC08}, that in the Witt case the decompositions \eqref{Eqn:Thm5.1OpPot} and \eqref{Eqn:Thm5.1OpPotI} are the same. \\

Now we describe the vanishing result for the index on $U_t$. Analogously as before, consider the operator  $\mathscr{D}^+_t\coloneqq  \mathscr{D}_{U_t,Q_{\geq}({A}(t))(H)}$ defined on the core
\begin{align*}
\mathcal{C}({D}^+_t)\coloneqq \{\sigma\in C^{1}_c((0,t],H^1))\:|\:Q_<({A}(t))\sigma(t)=0\}.
\end{align*}

\begin{theorem}\label{Thm:5.2Pot}
In the Witt case, for $t\in(0,t_0]$ sufficiently small we have 
\begin{align*}
\textnormal{ind}\left(\mathscr{D}^{+}_{U_t,Q_{\geq}({A}(t))(H)}\right)=0. 
\end{align*}
\end{theorem}

To prove this theorem we use a deformation argument, as in the proof of \cite[Theorem 5.2]{B09}, splitting the index into a contribution of the space of vertical harmonic forms and a contribution of the complement. More precisely, denote by $\Delta_V$ the fiber Laplacian and choose $\delta>0$ sufficiently small to define
\begin{align*}
P_\mathcal{H}(x)\coloneqq\frac{1}{2\pi i}\int_{|z|=\delta}(\Delta_{V,x}-z)^{-1}dz, \quad\text{for $x\in F$,}
\end{align*}
the projection onto the space of vertical harmonic $\mathcal{H}$ and let $P_{\mathcal{H}^\perp}\coloneqq I-P_\mathcal{H}$ be the complementary projection. 

\begin{lemma}[{\cite[Lemma 8.63]{JO17}}]\label{Lemma:CommPH}
The projection $I\otimes P_\mathcal{H}$ commutes with $\mathscr{A}_V$ and with the principal symbol of $A_H(t)$.
\end{lemma}

Let us define 
\begin{align*}
\mathscr{A}^{\delta}(r)\coloneqq &(I\otimes P_{\mathcal{H}})\mathscr{A}(r)(I\otimes P_{\mathcal{H}})+(I\otimes{P_{\mathcal{H}}}^\perp)\mathscr{A}(r)(I\otimes P_{\mathcal{H}}^\perp)\\
\eqqcolon &\mathscr{A}_{\mathcal{H}}(r)+\mathscr{A}_{{\mathcal{H}^\perp} }.
\end{align*}
By Lemma \ref{Lemma:CommPH} the difference $\mathscr{C}(r)\coloneqq \mathscr{A}^\delta(r)-\mathscr{A}(r)$
is a uniformly bounded norm, i.e. there exists $\mathscr{C}>0$ such that $\norm{\mathscr{C}(r)}\leq \mathscr{C}$ for all $r\in(0,t]$. We now consider the  {\em deformed} operator
\begin{align*}
\mathscr{D}^+_{t,\delta}\coloneqq \frac{\partial}{\partial r}+\mathscr{A}^{\delta}(r),
\end{align*}
defined on the core $\mathcal{C}(\mathscr{D}^+_{t,\delta})\coloneqq \{\sigma\in C^{1}_c((0,t],H^1))\:|\:Q_<(A^\delta(t))\sigma(t)=0\}$, where 
\begin{align*}
{A}^{\delta}(r)\coloneqq &(I\otimes P_{\mathcal{H}}){A}(r)(I\otimes P_{\mathcal{H}})+(I\otimes P_{\mathcal{H}}^\perp){A}(r)(I\otimes P_{\mathcal{H}}^\perp)\\
\eqqcolon &{A}_{\mathcal{H}}(r)+{A}_{{\mathcal{H}^\perp} }(r).
\end{align*}
Again, from Lemma \ref{Lemma:CommPH} it follows that $C(r)\coloneqq A^\delta(r)-A(r)$ has uniformly bounded operator. Since $\mathscr{D}^+_{t,\delta}-\mathscr{D}^+_{t}=\mathscr{C}(r)$ is uniformly bounded then, by \cite[Theorem 4.14]{BBC08}, 
\begin{align*}
\ind(\mathscr{D}^+_t )=\ind(\mathscr{D}^+_{t,\delta} )+\ind(Q_{<}(A(t))(H),Q_\geq(A^\delta(t))(H)).
\end{align*}
Now we introduce the operators 
\begin{align*}
\mathscr{D}^+_{t,\mathcal{H} /{\mathcal{H}^\perp} }\coloneqq\frac{\partial}{\partial r}+\mathscr{A}_{\mathcal{H} /{\mathcal{H}^\perp} }(r)
\end{align*}
defined on the core $\mathcal{C}(\mathscr{D}^+_{t,\mathcal{H} /{\mathcal{H}^\perp} })\coloneqq \{\sigma\in C^{1}_c((0,t],H^1))\:|\:Q_<(A_{\mathcal{H}/{\mathcal{H}^\perp} }(t))\sigma(t)=0\}$, which by orthogonality satisfy 
\begin{align*}
\ind(\mathscr{D}^+_{t,\delta})=\ind(\mathscr{D}^+_{t,\mathcal{H}})+\ind(\mathscr{D}^+_{t,{\mathcal{H}^\perp}})
\end{align*}
and implies the index relation
\begin{align*}
\ind(\mathscr{D}^+_t )=\ind(\mathscr{D}^+_{t,\mathcal{H} })+\ind(\mathscr{D}^+_{t,{\mathcal{H}^\perp} } )+\ind(Q_<(A(t))(H),Q_\geq(A^\delta(t))(H)). 
\end{align*}

\begin{remark}
In the Witt case, as $A(t)$ is invertible for $t>0$ small enough, it is easy to see that this is also the case for the operator $A^\delta(t)$. 
\end{remark}

Now we show these three contributions to the index vanish for $t>0$ sufficiently small. 
\begin{proposition}\label{Prop:DefIndex1}
For $t>0$ small enough, $\ind(Q_<(A(t))(H),Q_\geq(A^\delta(t))(H))=0$.
\end{proposition}

\begin{proof}
As for the Witt case we can assume $|A_V|\geq 1$, Lemma \ref{Lemma:VertOp1} implies
\begin{align*}
|A(t)|\geq\frac{\sqrt{C}}{t}|A_V|\geq\frac{\sqrt{C}}{t}.
\end{align*}
To apply the vanishing statement of Theorem \ref{BThm5.9} we require $\sqrt{C}/t \geq \sqrt{2}\norm{C(t)}$, 
which can always be achieved by making $t$ small,  since $C(t)$ is uniformly bounded. 
\end{proof}

\begin{proposition}\label{Prop:DefIndex2}
For $t>0$ small enough we have $\ind(\mathscr{D}^+_{t,{\mathcal{H}^\perp} } )=0$.
\end{proposition}

\begin{proof}
As we are in the Witt case we can assume $|A_V|\geq 1$. Hence, it is easy to see that for the operator
\begin{align*}
{D}^+_{t,{\mathcal{H}^\perp} } \coloneqq \frac{\partial}{\partial r}+A_{{\mathcal{H}^\perp}}(r),
\end{align*}
defined on the core $\mathcal{C}(\mathscr{D}^+_{t,{\mathcal{H}^\perp} })$,  we can prove an  analogue of Lemma \ref{Lemma:5.12}. That is, for $t>0$ small enough and $\sigma\in\mathcal{C}(\mathscr{D}^+_{t,{\mathcal{H}^\perp} })$, we can derive the estimate
\begin{align*}
\bnorm{D^+_{t,{\mathcal{H}^\perp}}\sigma}_{L^2((0,t],H)}\geq\frac{1}{2t}\bnorm{\left(A_{V,{\mathcal{H}^\perp}}+\frac{1}{2}\right)\sigma}_{L^2((0,t],H)},
\end{align*}
where $A_{V,{\mathcal{H}^\perp}}\coloneqq  (I\otimes P_{\mathcal{H}}^\perp)A_V(I\otimes P_{\mathcal{H}}^\perp)$. In particular, we can show as before that $\ind(D_{t,{\mathcal{H}^\perp}})=0$. Now, the strategy of the proof is to show that $\mathscr{D}^+_{t,{\mathcal{H}^\perp} }$
is a Kato perturbation of ${D}^+_{t,{\mathcal{H}^\perp} }$ on $\mathcal{C}(\mathscr{D}^+_{t,{\mathcal{H}^\perp} })$.  If we are able to proof an estimate of the form 
\begin{align}\label{Eqn:KatoAim}
\bnorm{\frac{\varepsilon}{2r}\sigma}_{L^2((0,t],H)}
\leq b\norm{D^+_{t,{\mathcal{H}^\perp}}\sigma}_{L^2((0,t],H)},
\end{align}
with $b<1$, then by \cite[Theorem IV.5.22, pg. 236]{KATO} it would follow that 
$\textnormal{ind}\left(\mathscr{D}^{+}_{t,\mathcal{H}\perp}\right)=0$. 
From the proof of Lemma \ref{Lemma:5.12}  we see that for an element $\sigma\in\mathcal{C}(\mathscr{D}^+_{t,{\mathcal{H}^\perp}})$ we have the norm identity (see \eqref{Eqn:AuxLemma5.12}),
\begin{align*}
\norm{{D}^+_{t,{\mathcal{H}^\perp}}\sigma(r) }_{H}^2=&\left(\norm{\sigma'(r)}^2_H-\frac{1}{4r^2}\norm{\sigma(r)}^2_H\right)+\frac{d}{dr}\inner{\sigma(r)}{{A}_{{\mathcal{H}^\perp}}(r)\sigma(r)}_H\\
&+r^{-2}\bnorm{\left({A}_{V,{\mathcal{H}^\perp}}+\frac{1}{2}\right)\sigma(r)}^2_H+r^{-1}\left\langle\left({A}_{V,{\mathcal{H}^\perp}}+\frac{1}{2}\right)\sigma(r),{L}^*\sigma(r)\right\rangle_H\notag\\
&-\inner{\sigma(r)}{A'_{H,{\mathcal{H}^\perp}}(0)\sigma(r)}_H+\norm{A_{H,{\mathcal{H}^\perp}}(r)\sigma(r)}_H^2.\notag
\end{align*}
Arguing as in the mentioned proof, we see that after integration between $0$ and $t$ the first term in brackets is non-negative by Hardy's inequality and the term containing the total derivative with respect to $r$ is also non-negative because $\sigma$ has compact support and because the boundary condition at $r=t$. Thus, it follows that for $t$ small enough and $0<\beta<1$,
\begin{align}\label{Eqn:Lemma5.12Modif}
\bnorm{\frac{1}{r}\left(A_{V,{\mathcal{H}^\perp}}+\frac{1}{2}\right)\sigma}_{L^2((0,t],H)}\leq (1+\beta)\norm{D^+_{t,{\mathcal{H}^\perp}}\sigma}_{L^2((0,t],H)}.
\end{align}

On the other hand, using the norm of the resolvent identity  (\cite[SectionV.5]{KATO}), we obtain the estimate
\begin{align*}
\norm{\sigma}
\leq \bnorm{\left({A}_{V,{\mathcal{H}^\perp}}+\frac{1}{2}\right)^{-1}}\bnorm{\left({A}_{V,{\mathcal{H}^\perp}}+\frac{1}{2}\right)\sigma(r)}
= \frac{1}{d^\perp} \bnorm{\left({A}_{V,{\mathcal{H}^\perp}}+\frac{1}{2}\right)\sigma(r)},
\end{align*}
where $d^\perp\coloneqq {\text{dist}(-1/2,\spec(A_{V,{\mathcal{H}^\perp}}))}$. All together, it follows from \eqref{Eqn:Lemma5.12Modif} that
\begin{align}\label{Eqn:ResultKatoEstimate}
\bnorm{\frac{\varepsilon}{2r}\sigma}_{L^2((0,t],H)}\leq \frac{(1+\beta)}{2d^\perp}\norm{D^+_t\sigma}_{L^2((0,t],H)},
\end{align}
i.e. we have shown the desired estimate with $b\coloneqq (1+\beta)/2d^\perp $. Recall we require the condition $b<1$, which translates to
$(1+\beta)<2d^\perp $. This can always be achieved by rescaling the vertical metric.
\end{proof}

\begin{proposition}\label{Prop:DefIndex3}
For $t$ small enough we have $\ind(\mathscr{D}^+_{t,{\mathcal{H}} } )=0$.
\end{proposition}

\begin{proof}
We will proceed similarly as in the proof of \cite[Theorem 5.2]{B09}. Since we can identify the first order part of $\mathscr{A}_H$ with the odd signature operator $A_F$ of $F$  with coefficients in $\mathcal{H}$ we know, by the discussion of \cite[Remark (3),(4), pg. 61]{APSI}, that in the Witt case  there exists a self-adjoint involution $\mathscr{U}$ such that $\mathscr{U}A_F\mathscr{U}=-A_F$. In particular, $\mathscr{U}$ anti-commutes with the principal symbol of $\mathscr{A}_H$ and therefore
$\mathscr{U}\mathscr{A}_{\mathcal{H}}(r)\mathscr{U}=-{\mathscr{A}}_\mathcal{H}(r)+\mathscr{C}_2(r)$,
where  $\norm{\mathscr{C}_2(r)}_H\leq \mathscr{C}_2$ for $r\in(0,t]$.  Similarly,  we get an analogous formula for $A_\mathcal{H}$ since it has the same principal symbols as $\mathscr{A}_\mathcal{H}$, i.e.
$\mathscr{U}{A}_{\mathcal{H}}(r)\mathscr{U}=-{A}_\mathcal{H}(r)+{C}_2(r)$, 
where $\norm{{C}_2(r)}_H\leq C_2$ for $r\in(0,t]$. Observe now that the operator $\mathscr{U}\mathscr{D}^+_t \mathscr{U}$ defined on the core
$\mathcal{C}(\mathscr{U}\mathscr{D}^+_{t,\mathcal{H}}\mathscr{U})\coloneqq \{\sigma\in C^{1}_c((0,t],H^1))\:|\:Q_<(\mathscr{U} A_{\mathcal{H}} (t)\mathscr{U})\sigma(t)=0\}$
is given by 
\begin{align*}
\left(\frac{\partial }{\partial r}+\mathscr{U} \mathscr{A}_H(r)\mathscr{U} \right)\sigma(r)=-\left(-\frac{\partial }{\partial r}+\mathscr{A}_\mathcal{H}(r)-\mathscr{C}_2(r)\right)\sigma(r). 
\end{align*}
As we are in the Witt case we can use Lemma \ref{Lemma:VertOp1} and \cite[Theorem V.4.10]{KATO} to verify for $t>0$ sufficiently small the relation $Q_<(-A_{\mathcal{H}} (t)+C_2(t))=Q_<(-A_{\mathcal{H}})=Q_>(A_{\mathcal{H}})$ since $C_2(r)$ is uniformly bounded. We now compare $\mathscr{U} \mathscr{D}^+_t \mathscr{U} $ with the adjoint operator
\begin{align*}
(\mathscr{D}^+_{t,\mathcal{H}})^*\sigma(r)=\left(-\frac{\partial }{\partial r}+\mathscr{A}_H(r)\right)\sigma(r),
\end{align*}
defined on the core
$\mathcal{C}((\mathscr{D}^+_{t,\mathcal{H}})^*)\coloneqq \{\sigma\in C^{1}_c((0,t],H^1))\:|\:Q_\geq(A_{\mathcal{H}} (t))\sigma(t)=0\}$.
Since in the Witt case $A_\mathcal{H}$ is invertible then $Q_\geq(A_{\mathcal{H}} (t))=Q_>(A_{\mathcal{H}} (t))$, from where it follows that
$\mathscr{U} \mathscr{D}^+_t \mathscr{U} =-(\mathscr{D}^+_{t,\mathcal{H}})^*$. Altogether,
\begin{align*}
\ind(\mathscr{D}^+_{t,{\mathcal{H}} } )=\ind( \mathscr{U}\mathscr{D}^+_{t,{\mathcal{H}}} \mathscr{U})=(\mathscr{D}^+_{t,\mathcal{H}})^*=-\ind(\mathscr{D}^+_{t,{\mathcal{H}} } ), 
\end{align*}
which shows that $\ind(\mathscr{D}^+_{t,{\mathcal{H}} } )=0$. 
\end{proof}

\begin{proof}[Proof of Theorem \ref{Thm:5.2Pot}]
This follows now from the deformation argument described above, the vanishing results of Proposition \ref{Prop:DefIndex1}, Proposition \ref{Prop:DefIndex2} and Proposition \ref{Prop:DefIndex3}.
\end{proof}

Regarding the index contribution of $Z_t$ it is easy to see that, by deforming the metric close to $r=t$, we can adapt the proof Theorem \ref{Thm:5.3} to get the analogous formula,
\begin{align*}
\textnormal{ind}\left(\mathscr{D}^+_{Z_t,Q_{<}({A}(t))(H)}\right)&=\textnormal{ind}\left({D}^+_{
Z_t,Q_{<}(A_0(t))(H)}\right)+(Q_{<}({A}(r))(H),Q_{\geq }({A}_0(r))(H)), 
\end{align*}
where ${D}^+_{Z_t,Q_{<}(A_0(t))(H)}$ is the signature operator on the manifold with boundary $Z_t$. In particular, we can use  \eqref{IndS1SignKern} to conclude that for  $t>0$ small enough 
\begin{align*}
\textnormal{ind}\left(\mathscr{D}^+_{Z_t,Q_{<}({A}(t))(H)}\right)=\sigma_{S^1}(M)-\frac{1}{2}\dim\ker(A_0(t))+(Q_{<}({A}(r))(H),Q_{\geq }({A}_0(r))(H)).
\end{align*}

Finally, from this formula and from the vanishing Theorem \ref{Thm:5.2Pot} we obtain a partial answer for the index of the operator $\mathscr{D}^+$. 

\begin{theorem}
Let $M$ be a closed, oriented Riemannian $4k+1$ dimensional manifold on which $S^1$ acts effectively and semi-freely by orientation preserving isometries. If the codimension of the fixed point set $M^{S^1}$ in $M$ is divisible by four, then $M/S^1$ is a Witt space and the index of the graded Dirac-Schr\"odinger operator is 
\begin{align*}
\ind(\mathscr{D}^+)=\sigma_{S^1}(M)=\int_{M_0/S^1}L\left(T(M_0/S^1),g^{T(M_0/S^1)}\right).
\end{align*}
\end{theorem}

 \bibliographystyle{acm}
\bibliography{references} 
\end{document}